\documentclass[11pt]{article}
\usepackage{amssymb,amsmath,amsfonts,amsthm,rotating,setspace,graphicx,float}
\usepackage{enumerate}
\usepackage{cases}

\numberwithin{equation}{section}

\newcommand{\vp}{\varepsilon}

\def\R{{\mathbb R} }

\theoremstyle{plain}
\newtheorem{thm}{Theorem}[section]
\newtheorem{lem}{Lemma}[section]
\newtheorem{pro}{Proposition}[section]

\theoremstyle{definition}

\newtheorem{rem}{Remark}[section]

\begin{document}

\title{
 Pointwise Bounds and Blow-up for Systems of
  Semilinear Parabolic 
  Inequalities and Nonlinear Heat Potential
  Estimates}

\author{Marius Ghergu\footnote{School of Mathematical Sciences,
    University College Dublin, Belfield, Dublin 4, Ireland; {\tt
      marius.ghergu@ucd.ie}},\quad
Steven D.~Taliaferro\footnote{Mathematics Department, Texas A\&M
    University, College Station, TX 77843-3368; {\tt stalia@math.tamu.edu}} 
\footnote{Corresponding author, Phone 001-979-845-3261, Fax
  001-979-845-6028}}

\date{}
\maketitle

\begin{abstract}
We study the  behavior for $t$ small and positive of
$C^{2,1}$ nonnegative solutions $u(x,t)$ and $v(x,t)$ of the system
\begin{equation*}
\begin{aligned}
  0\leq u_t-\Delta u\leq v^\lambda  \\  
 0\leq v_t-\Delta v\leq u^\sigma
\end{aligned} \qquad \mbox{ in } \Omega\times (0,1),
\end{equation*}
where $\lambda$ and $\sigma$ are nonnegative constants and $\Omega$ is
an open subset of $\R^n$, $n\ge 1$.
We provide optimal conditions on $\lambda$ and $\sigma$  such that
solutions of this system satisfy pointwise bounds in compact subsets
of $\Omega$ as $t\to 0^+$. Our approach relies on new pointwise bounds
for nonlinear heat potentials which are the parabolic analog of
similar bounds for nonlinear Riesz potentials.
\medskip

\medskip

\end{abstract}

\section{Introduction}\label{sec1}
In this paper we study the behavior for $t$ small and positive of
$C^{2,1}$ nonnegative solutions $u(x,t)$ and $v(x,t)$ of the system
\begin{equation}\label{eq1.1}
\begin{aligned}
  0\leq u_t-\Delta u\leq v^\lambda  \\  
 0\leq v_t-\Delta v\leq u^\sigma
\end{aligned} \qquad \mbox{ in } \Omega\times (0,1),
\end{equation}
where $\lambda$ and $\sigma$ are nonnegative constants and $\Omega$ is
an open subset of $\R^n$, $n\ge 1$.
More precisely, we consider the following question.
\medskip

\noindent {\bf Question 1}. For which nonnegative constants $\lambda$
and $\sigma$ do there exist continuous functions $h_1 ,h_2
:(0,1)\to(0,\infty)$ such that for all compact subsets $K$ of $\Omega$
and for all $C^{2,1}$ nonnegative solutions $u(x,t)$
 and
$v(x,t)$ of the system \eqref{eq1.1} we have
\begin{equation}\label{eq1.3}
 \max_{x\in K}u(x,t)=O(h_1 (t)) \quad \text{ as } t\to 0^+
\end{equation}
\begin{equation}\label{eq1.4}
 \max_{x\in K}v(x,t))=O(h_2 (t)) \quad \text{ as }t\to 0^+
\end{equation}
and what are the optimal such $h_1$ and $h_2$ when they exist?
\medskip

We call a function $h_1$ (resp. $h_2$) with the above properties a
pointwise bound in compact subsets for $u$ (resp. $v$) as $t\to0^+$.

\begin{rem}\label{rem1}
 Let
 \begin{equation}\label{eq1.5}
  \Phi(x,t)=\begin{cases}
\frac{1}{(4\pi t)^{n/2}} 
e^{-\frac{|x|^2}{4t}}&\text{for }(x,t)\in \R^n\times (0,\infty)\\
   0&\text{for }(x,t)\in \R^n\times (-\infty,0]
   \end{cases}             
  \end{equation}
  be the heat kernel. Since $\Phi_t-\Delta\Phi=0$ in $\R^n\times(0,\infty)$, the
  functions $u_0=v_0=\Phi$ are always $C^{2,1}$ nonnegative solutions
  of \eqref{eq1.1}.  Hence, since $\Phi(0,t)=\frac{1}{(4\pi
    t)^{n/2}}$, any pointwise bound as $t\to 0^+$ in compact subsets
  of $\Omega$ for nonnegative solutions of \eqref{eq1.1} must be at
  least as large as $t^{-n/2}$ and whenever $t^{-n/2}$ is such a bound
  for $u$ (resp. $v$) it is necessarily optimal. In this case we say
  that $u$ (resp. $v$) is {\it heat bounded} in compact subsets of
  $\Omega$ as $t\to 0^+$.
\end{rem}

We shall see that whenever a pointwise bound as $t\to 0^+$ in compact
subsets of $\Omega$ for nonnegative solutions of
\eqref{eq1.1} exists, then $u$ or $v$ (or both) are heat bounded
as $t\to 0^+$.

The literature on scalar and systems of parabolic {\it equations} is
quite vast. A good source for this material is the book
\cite{QS2007}. However, very little attention has been paid to
systems of parabolic {\it inequalities}, and, as far as we know, all
results deal with a very different aspect of these inequalities; namely the
nonexistence of global solutions. See for example
\cite{C2000,C2001,MP2004}.

Let us mention some of the methods and tools we use to study Question
1. First and most noteworthy of these are some new results for linear
and nonlinear heat potentials. To motivate them recall that if
$f:\R^n\to\R$, $n\ge 3$, is a nonnegative measurable function,
$\alpha\in(0,n)$ is a constant, and 
\[
\Gamma(x)=\frac{C(n)}{|x|^{n-2}}
\]
is a fundamental solution of $-\Delta$ in $\R^n$ then the Riesz
potential of $f$ is given by the convolution 
\[
\Gamma^{\frac{n-\alpha}{n-2}}\ast f.
\]
It has been extensively studied because of its usefulness in potential
theory and the study of elliptic PDEs. See for example the books
\cite{Stein,AH,M}. Three important results concerning the Riesz
potential operator, which are relevant to this paper, are Hedberg's
inequality \cite{H1972}; the Hardy-Littlewood-Sobolev inequality (see
\cite[p. 119]{Stein}); and estimates for the nonlinear potential 
\[
\Gamma^{\frac{n-\alpha}{n-2}}
\ast\left(\left(\Gamma^{\frac{n-\beta}{n-2}}\ast f\right)^\sigma\right)
\]
first studied in \cite{MH}. A crucial tool for the proofs of these
results is the celebrated Hardy-Littlewood maximal function
inequalities (see \cite[p. 5]{Stein}).

In our study of Question 1 there arises naturally the need to obtain
similar results for the convolution
\begin{equation}\label{eq1.6}
\Phi^{\frac{n+2-\alpha}{n}}\ast f,
\end{equation}
where $f:\R^n\times\R\to\R$, $n\ge 1$, is a nonnegative measurable function,
$\alpha\in(0,n+2)$ is a constant, and $\Phi$
is the fundamental solution of the heat operator given by
\eqref{eq1.5}. These new results for the heat potential operator \eqref{eq1.6}
are stated and proved in Section \ref{sec3} using a modified version
of the Hardy-Littlewood maximal function inequalities in which Euclidean balls
in $\R^n$ are replaced with heat balls in $\R^n\times\R$. 

Two other tools required are a Moser type iteration (see Lemma
\ref{lem4.6}) and a representation formula given in Lemma \ref{lem4.1}
for nonnegative super temperatures which is the parabolic analog of
the Brezis-Lions representation formula \cite{BL} for nonnegative
superharmonic functions.

\section{Statement of results}\label{sec2}

In this section we state our results for Question 1.   We can assume
without loss of generality that $\sigma\leq \lambda$.

If $\lambda$ and $\sigma$ are nonnegative constants satisfying
$\sigma\leq \lambda$ then $(\lambda ,\sigma)$ belongs to one of the
following four pointwise disjoint subsets of the $\lambda
\sigma$-plane:
\begin{align*}
 &A:=\left\{ (\lambda ,\sigma): \, 0\leq \sigma \leq \lambda \leq 
\frac{n+2}{n}\right\} \\
 &B:=\left\{ (\lambda ,\sigma): \, \lambda>\frac{n+2}{n}
   \quad\text{and}
\quad 0\leq \sigma<\frac{2}{n}+\frac{n+2}{n\lambda} \right\} \\
 &C:=\left\{ (\lambda ,\sigma): \, \lambda>\frac{n+2}{n} \quad
   \text{and}\quad \frac{2}{n}+\frac{n+2}{n\lambda} <\sigma \leq \lambda \right\} \\
 &D:=\left\{ (\lambda ,\sigma): \, \lambda>\frac{n+2}{n} \quad
   \text{and}\quad \sigma=\frac{2}{n}+\frac{n+2}{n\lambda} \right\}.
\end{align*}
\begin{figure}[H]
 \includegraphics[scale=.65]{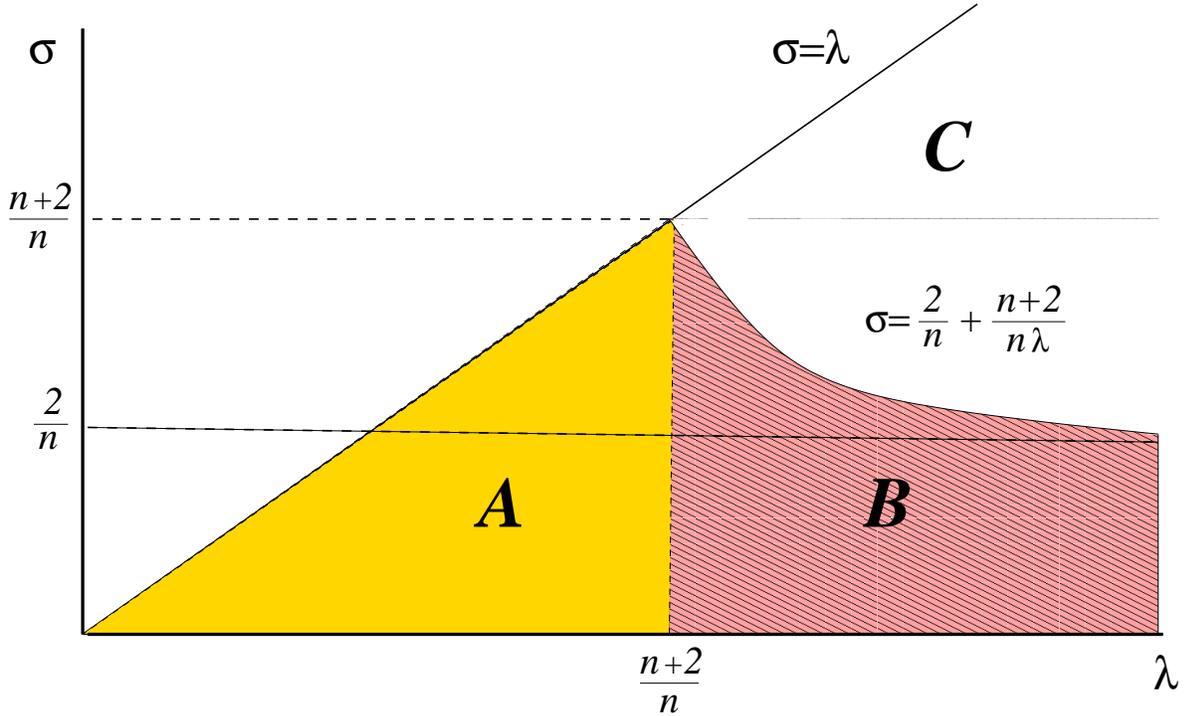}
 \caption{Graph of regions A, B and C.}
\end{figure}

Note that $A$, $B$ and $C$ are two dimensional regions in the $\lambda
\sigma$-plane whereas $D$ is the curve separating $B$ and $C$. (See
Figure 1.)

In this section we give a complete answer to Question 1 when
$(\lambda ,\sigma)\in A\cup B\cup C$.  The following theorem deals
with the case that $(\lambda ,\sigma)\in A$.

\begin{thm}\label{thm2.1}
Suppose $u(x,t)$ and $v(x,t)$ are $C^{2,1}$ nonnegative solutions of the system
\begin{equation}\label{eq2.1}
\begin{aligned}
0\leq u_t-\Delta u\leq \left(v+\left(\frac{1}{\sqrt{t}}\right)^n\right)^\lambda\\
0\leq v_t-\Delta v\leq \left(u+\left(\frac{1}{\sqrt{t}}\right)^n\right)^\sigma
\end{aligned} \qquad \mbox{ in } \Omega\times (0,1),
\end{equation}
where the constants $\lambda$ and $\sigma$ satisfy
\begin{equation}\label{eq2.3}
0\le\sigma\le\lambda\le\frac{n+2}{n}
\end{equation}
and $\Omega$ is an open subset of $\R^n$, $n\ge 1$.
Then both $u$ and $v$ are heat bounded in compact subsets of $\Omega$
as $t\to0^+$, that is, for each compact subset $K$ of $\Omega$ we have
 \begin{equation}\label{eq2.6}
  \max_{x\in K}u(x,t)=O\left(\left(\frac{1}{\sqrt{t}}\right)^n\right) 
\quad \text{ as }t\to0^+
 \end{equation}
and
 \begin{equation}\label{eq2.7}
  \max_{x\in K}v(x,t)=O\left(\left(\frac{1}{\sqrt{t}}\right)^n\right) 
\quad \text{ as }t\to0^+.
 \end{equation}
\end{thm}

By Remark \ref{rem1}, the bounds \eqref{eq2.6} and \eqref{eq2.7} are optimal.

The following two theorems deal with the case that $(\lambda ,\sigma) \in B$.

\begin{thm}\label{thm2.2}
Suppose $u(x,t)$ and $v(x,t)$ are $C^{2,1}$ nonnegative solutions of the system
\eqref{eq2.1}
where the constants $\lambda$ and $\sigma$ satisfy
\begin{equation}\label{eq2.8}
\lambda>\frac{n+2}{n}\quad\text{and}\quad
  \sigma<\frac{2}{n}+\frac{n+2}{n\lambda}
\end{equation}
and $\Omega$ is an open subset of $\R^n$, $n\ge 1$. Then for each
compact subset $K$ of $\Omega$ we have
 \begin{equation}\label{eq2.9}
  \max_{x\in K}u(x,t)
=o\left(\left(\frac{1}{\sqrt{t}}\right)^{\frac{n^2\lambda}{n+2}}\right) 
\quad \text{ as }t\to0^+
 \end{equation}
and
 \begin{equation}\label{eq2.10}
  \max_{x\in K}v(x,t)=O\left(\left(\frac{1}{\sqrt{t}}\right)^n\right) 
\quad \text{ as }t\to0^+.
 \end{equation}
\end{thm}

By the following theorem the bounds \eqref{eq2.9} and \eqref{eq2.10}
for $u$ and $v$ in Theorem \ref{thm2.2} are optimal.

\begin{thm}\label{thm2.3}
  Suppose $\lambda$ and $\sigma$ satisfy \eqref{eq2.8} and
  $\varphi:(0,1)\to(0,1)$ is a continuous function satisfying $\lim_{t\to
    0^{+}} \varphi(t)=0$.  Then there exist $C^\infty$ positive solutions
  $u(x,t)$ and $v(x,t)$ of the system
 \begin{equation}\label{eq2.11}
 \begin{aligned}
 0 & \le u_t-\Delta u\le v^\lambda\\
 0 & \le v_t-\Delta v\le u^\sigma
 \end{aligned}
 \text{\qquad in }\  (\R^n\times\R)\setminus \{(0,0)\}, \ n\geq 1,
 \end{equation}
 such that
\begin{equation}\label{eq2.13}
  u(0,t)\neq O\left(
    \varphi(t)\left(\frac{1}{\sqrt{t}}\right)^{\frac{n^2\lambda}{n+2}}\right)  
\quad \text{as }t\to0^+
 \end{equation}
 and
 \begin{equation}\label{eq2.14}
  \liminf_{t\to 0^+}v(0,t)t^{n/2} >0.
 \end{equation}
\end{thm}

The following theorem deals with the case that $(\lambda ,\sigma)\in
C$.  In this case there exist pointwise bounds for neither $u$ nor
$v$.

\begin{thm}\label{thm2.4}
 Suppose $\lambda$ and $\sigma$ are constants satisfying
 \begin{equation}\label{eq2.15}
  \lambda>\frac{n+2}{n}\quad\text{and}\quad
\frac{2}{n}+\frac{n+2}{n\lambda} <\sigma\leq\lambda.
 \end{equation}
 Let $\varphi:(0,1)\to(0,\infty)$ be a continuous function satisfying
 \begin{equation*}
  \lim_{t\to0^{+}} \varphi(t)=\infty.
 \end{equation*}
 Then there exist $C^\infty$ solutions $u(x,t)$ and $v(x,t)$ of the system
 \begin{equation}\label{eq2.16}
 \left. 
 \begin{aligned}
 0 & \le u_t-\Delta u\le v^\lambda\\
 0 & \le v_t-\Delta v\le u^\sigma\\
 u & >1, \, v>1
 \end{aligned}
 \right\}
 \text{\qquad in }\ (\R^n\times\R)\setminus \{(0,0)\}, \ n\geq 1, 
 \end{equation}
 such that
 \begin{equation}\label{eq2.17}
  u(0,t)\neq O(\varphi(t)) \quad \text{as }t\to0^+
 \end{equation}
 and
 \begin{equation}\label{eq2.18}
  v(0,t)\neq O(\varphi(t))\quad \text{as }t\to0^+.
 \end{equation}
\end{thm}

The following theorem can be viewed as the limiting case of Theorem
\ref{thm2.2} as $\lambda\to\infty$.

\begin{thm}\label{thm2.5}
Suppose $u(x,t)$ and $v(x,t)$ are $C^{2,1}$ nonnegative solutions of the system
\[
\begin{aligned}
0&\le u_t-\Delta u\\
0&\le v_t-\Delta v\le \left(u+\left(\frac{1}{\sqrt{t}}\right)^n\right)^\sigma
\end{aligned} \qquad \mbox{ in } \Omega\times (0,1), 
\]
where $\sigma<2/n$
and $\Omega$ is an open subset of $\R^n$, $n\ge 1$. Then $v$ is heat bounded in compact subsets of $\Omega$
as $t\to0^+$, that is, for each compact subset $K$ of $\Omega$ we have
 \begin{equation}\label{eq2.19}
  \max_{x\in K}v(x,t)=O\left(\left(\frac{1}{\sqrt{t}}\right)^n\right) 
\quad \text{ as }t\to0^+.
 \end{equation}
\end{thm}

By Remark \ref{rem1}, the bound \eqref{eq2.19} is optimal.

A consequence of the methods we use to prove the results in this
section is the following simple, optimal, and apparently unknown
result. The proof is, however, nontrivial being based on the
representation formula in Lemma \ref{lem4.1}.

\begin{thm}\label{thm2.6}
Suppose $u$ is a $C^{2,1}$ nonnegative solution of 
\[
0\le u_t-\Delta u\le\left(\frac{1}{\sqrt{t}}\right)^\gamma \quad
\text{in }\Omega\times(0,1),
\]
where $\gamma\in\R$ and $\Omega$ is an open subset of $\R^n$, $n\ge 1$.
Then for each compact subset $K$ of $\Omega$ we have
\[
\max_{x\in K}u(x,t)
=\begin{cases}
o\left(\left(\frac{1}{\sqrt t}\right)^{\gamma\frac{n}{n+2}}\right)
&\text{if $\gamma>n+2$}\\
O\left(\left(\frac{1}{\sqrt t}\right)^n\right)
&\text{if $\gamma\le n+2$}
\end{cases}
\quad\text{as }t\to 0^+.
\]
\end{thm}

In the next section we shall derive the analog of Hedberg's
inequality for heat potentials as well as estimates for nonlinear
heat potential which are crucial tools of our approach. More
preliminary results are provided in Section \ref{sec4}. The proof of
the main results will be given in Section~\ref{sec5}.

\section{Hedberg's inequality for heat potentials and nonlinear heat
  potential estimates}\label{sec3}

We define $J_\alpha:\R^n\times\R\to[0,\infty)$ for $n\ge 1$ and
$0<\alpha<n+2$ by 
\begin{equation}\label{eq 3.1}
J_\alpha(x,t)=\Phi(x,t)^{\frac{n+2-\alpha}{n}}
\end{equation}
where $\Phi$ is the heat kernel \eqref{eq1.5}. If $f:\R^n\times\R\to\R$ is
a nonnegative measurable function then we call the
convolution $J_\alpha\ast f:\R^n\times\R\to [0,\infty]$, given by 
\begin{equation}\label{eq3.2}
(J_\alpha\ast f)(x,t)
=\iint_{\R^n\times\R}\Phi(x-y,t-s)^{\frac{n+2-\alpha}{n}}f(y,s)\,dy\,ds, 
\end{equation}
a heat potential of $f$.

The main result in this section is the following theorem which gives
estimates for the nonlinear potential $J_\alpha \ast ((J_\beta\ast
f)^\sigma)$. This potential is the nonlinear heat potential analog of the
nonlinear Riesz potential first studied by Maz'ya and Havin
\cite{MH}. See also \cite[Chapter 10]{M}.

\begin{thm}\label{nonlinear}
  Suppose 
\begin{equation}\label{assumptions}
\alpha,\beta\in (0,n+2), \quad \sigma>\frac{\alpha}{n+2-\beta},
\quad \text{and}\quad 1\le r<\frac{(n+2)\sigma}{\alpha+\beta\sigma}.
\end{equation}
Then there exists a constant
  $C=C(n,\alpha,\beta,\sigma, r)>0$ such that for all nonnegative
  measurable functions $f:\R^n\times\R\to\R$, $n\ge 1$, we have
$$ \| J_\alpha\ast((J_\beta\ast f)^\sigma)\|_{L^\infty(\R^n\times\R)}
\leq C\|f\|_{L^r(\R^n\times\R)}^{\frac{(\alpha+\beta\sigma)r}{n+2}}
  \|f\|_{L^\infty(\R^n\times\R)}^{\frac{\sigma(n+2-\beta r)-\alpha r}{n+2}}.
$$
\end{thm}

For the proof of Theorem \ref{nonlinear}, we will need three auxiliary
results of independent interest. Namely, (i) a heat potential analog
of Hedberg's Riesz potential inequality; (ii) a heat ball analog of
the Hardy-Littlewood maximal function inequality; and (iii) a new
Sobolev inequality for heat potentials. These three results are stated
below in Theorems \ref{thm3.2}, \ref{thm3.3}, and \ref{thm3.4},
respectively.

To precisely state these results we first need
some definitions.

If $(x,t)\in\R^n \times\R$, $n\geq 1$ and $r\in[0,\infty]$ then the set
\begin{equation}\label{heatball}
 E_r (x,t):=
 \begin{cases}
  \{(y,s)\in\R^n \times\R:\Phi(x-y,t-s)>\frac{1}{r^n}\} & \text{if }0<r<\infty,\\
  \{(y,s)\in\R^n \times\R:s<t\} & \text{if }r=\infty,\\
  \emptyset & \text{if }r=0
 \end{cases}
\end{equation}
is called a {\it heat ball}.

Let $d$ be the metric on $\R^n \times\R$ defined by 
\begin{equation}\label{d-metric}
 d\left((x,t),(y,s)\right)=\max\left\{|x-y|,\sqrt{|t-s|}\right\}.  
\end{equation}
For $(x,t)\in\R^n \times\R$ and $r>0$ let
\begin{equation}\label{d-ball}
 Q_r (x,t)=\{(y,s)\in\R^n \times\R:d\left((x,t),(y,s)\right)<r\}
\end{equation}
be the open ball in the metric space $(\R^n \times\R,d)$ with center
$(x,t)$ and radius $r$, and let
\begin{equation}
 P_r (x,t)=\{(y,s)\in Q_r(x,t):s<t\}
\end{equation}
 be the lower half of $Q_r(x,t)$.

Under the change of variables
\begin{equation}\label{eq3.4}
 -r\eta=x-y,\quad -r^2 \zeta=t-s,\quad\text{ where }r>0,
\end{equation}
we have
\begin{equation}\label{eq3.5}
 \Phi(x-y,t-s)=r^{-n}\Phi(-\eta,-\zeta).
\end{equation}
Thus for $0\leq a<b\leq\infty$ and $\beta\in\R$ we find that
\begin{equation}\label{eq3.6}
 \iint_{E_{br}(x,t)\backslash E_{ar}(x,t)}\Phi(x-y,t-s)^\beta \,dy\,ds=r^{n+2-n\beta}\iint_{E_b (0,0)\backslash E_a (0,0)}\Phi(-\eta,-\zeta)^\beta \,d\eta\,d\zeta
\end{equation}
and using the fact that $$\int_{\R^n}\Phi(-\eta,-\zeta)^\beta
\,d\eta=\frac{C(\beta,n)}{(-\zeta)^{(\beta-1)n/2}}\quad 
\text{for }\zeta<0 \text{ and } \beta>0,$$ 
it is easy to check, for use in \eqref{eq3.6}, that
\begin{equation}\label{eq3.7}
 \iint_{E_\infty (0,0)\backslash E_1 (0,0)}\Phi(-\eta,-\zeta)^\beta \,d\eta\,d\zeta<\infty\quad(\text{resp. }\iint_{E_1 (0,0)}\Phi(-\eta,-\zeta)^\beta d\eta\,d\zeta<\infty)
\end{equation}
if $\beta>\frac{n+2}{n}\,(\text{resp. }\beta<\frac{n+2}{n})$.  Clearly
\begin{equation}\label{eq3.8}
 |Q_r (x,t)|=r^{n+2}|Q_1 (0,0)|\quad\text{and}\quad
|P_r (x,t)|=r^{n+2}|P_1 (0,0)|
\end{equation}
and taking $\beta=0=a$ and $b=1$ in \eqref{eq3.6} we get
\begin{equation}\label{eq3.9}
 |E_r (x,t)|=r^{n+2}|E_1 (0,0)|.
\end{equation}

\begin{lem}\label{lem3.1}
 There exists $r_0 =r_0 (n)>0$ such that 
\[
E_r (x,t)\subset P_{r_0 r}(x,t)\subset Q_{r_0 r}(x,t)\quad\text{for all } 
(x,t)\in\R^n \times\R \text{ and all }r>0.
\]
\end{lem}

\begin{proof}
 Choose $r_0 >0$ such that $E_1 (0,0)\subset Q_{r_0}(0,0)$.  Suppose $r>0 ,\,(x,t)\in\R^n \times\R$, and
 \begin{equation}\label{eq3.10}
  (y,s)\in E_r (x,t).
 \end{equation}
 Then $s<t$ and making the change of variables \eqref{eq3.4} we have
 \eqref{eq3.5} holds.  It follows therefore from the definition of
 $E_r (x,t)$ and \eqref{eq3.10} that
\[
(\eta,\zeta)\in E_1 (0,0)\subset Q_{r_0}(0,0).
\]  
Hence
$\max\{|\eta|,\sqrt{-\zeta}\}<r_0$.  So 
\[
d\left((x,t),(y,s)\right)=\max\{|x-y|,\sqrt{t-s}\}
=\max\{r|\eta|,r\sqrt{-\zeta}\}<r_0r.
\] 
Thus $(y,s)\in Q_{r_0 r}(x,t)$. Hence, since $s<t$, 
we have $(y,s)\in P_{r_0r}(x,t)$.
\end{proof}

\begin{lem}\label{lem3.2}
  Suppose $a>-1$, $b\ge 0$, and $\alpha>0$ are constants and
  $g:\R^{m}\to\R,\,m\geq 1$, is a nonnegative measurable function.
  Then
 \begin{equation}\label{eq3.11}
  \frac{\alpha^{a+1}}{a+1}\int_{\R^m}g(x)^{a+b+1}dx=\int^{\infty}_{0}
\lambda^a \left(\int_{\{g>\lambda/\alpha\}}g(x)^b dx\right)d\lambda.
 \end{equation}
\end{lem}
\begin{proof}
 \begin{align*}
  \text{L.H.S. of \eqref{eq3.11} }&=\int_{\R^m}\frac{(\alpha g(x))^{a+1}}{a+1}g(x)^b dx\\
  &=\int_{\R^m}\left(\int^{\alpha g(x)}_{0}\lambda^a d\lambda\right)g(x)^b dx\\
  &=\int_{\R^m}\left(\int^{\infty}_{0}(\chi_{[0,\alpha g(x)]}(\lambda))\lambda^a d\lambda\right)g(x)^b dx\\
  &=\int^{\infty}_{0}\lambda^a \left(\int_{\R^m}\chi_{[0,\alpha g(x)]}(\lambda)g(x)^b dx\right)d\lambda\\
  &=\text{R.H.S. of \eqref{eq3.11}.}
 \end{align*}
\end{proof}

The following theorem is the heat potential analog of Hedberg's Riesz
potential inequality \cite{H1972}.

\begin{thm}\label{thm3.2}
 Suppose $0<\alpha<n+2$ and $1\leq p<\frac{n+2}{\alpha}$ are constants and $f:\R^n \times\R\to\R$ is a nonnegative measurable function.  Then
 \begin{equation}\label{eq3.12}
  J_\alpha *f(x,t)\leq C\| f\|^{\frac{\alpha p}{n+2}}_{L^p(\R^n\times\R)}(Mf(x,t))^{1-\frac{\alpha p}{n+2}}\quad\text{for }(x,t)\in\R^n \times\R
 \end{equation}
 where $C=C(n,\alpha,p)$ is a positive constant and 
 \begin{equation}\label{eq3.13}
  Mf(x,t)=\sup_{r>0}\frac{1}{|E_r (x,t)|}\iint_{E_r (x,t)}f(y,s)\,dy\,ds
 \end{equation}
 is the heat ball analog of the Hardy-Littlewood maximal function.
\end{thm}
\begin{proof}
 Let $\rho>0$.  Then
 \begin{align*}
  \int^{\rho}_{0}&\frac{1}{r^{n+3-\alpha}}\left(\iint_{E_r (x,t)}f(y,s)\,dy\,ds\right)dr\\
  &=\int^{\rho}_{0}\left(\iint_{\R^n \times\R}\frac{1}{r^{n+3-\alpha}}\chi_{E_r (x,t)}(y,s)f(y,s)\,dy\,ds\right)dr\\
  &=\iint_{\R^n \times\R}\left(\int^{\rho}_{0}\frac{1}{r^{n+3-\alpha}}\chi_{E_r (x,t)}(y,s)\,dr\right)f(y,s)\,dy\,ds\\
  &=\iint_{E_\rho (x,t)}\left(\int^{\rho}_{0}\frac{1}{r^{n+3-\alpha}}\chi_{E_r (x,t)}(y,s)\,dr\right)f(y,s)\,dy\,ds\\
  &=\iint_{E_\rho (x,t)} \left(\int^{\rho}_{\frac{1}{\Phi(x-y,t-s)^{1/n}}}\frac{dr}{r^{n+3-\alpha}}\right)f(y,s)\,dy\,ds\\
  &=\frac{1}{n+2-\alpha}\iint_{E_\rho (x,t)}\left(\Phi(x-y,t-s)^{\frac{n+2-\alpha}{n}}-\frac{1}{\rho^{n+2-\alpha}}\right)f(y,s)\,dy\,ds.
 \end{align*}
 It follows therefore from \eqref{eq3.9} and \eqref{eq3.13} that 
 \begin{equation}\label{eq3.14}
  \iint_{E_\rho (x,t)}\Phi(x-y,t-s)^{\frac{n+2-\alpha}{n}}f(y,s)\,dy\,ds\leq C\rho^\alpha Mf(x,t)\quad\text{for }(x,t)\in\R^n \times\R
 \end{equation}
 where $C=C(n,\alpha)$ is a positive constant.
 
 Let $q$ be the conjugate H\"older exponent of $p$.  Then 
\[
\frac{1}{q}=1-\frac{1}{p}<1-\frac{\alpha}{n+2}=\frac{n+2-\alpha}{n+2}
\] and thus $\frac{n+2-\alpha}{n}q>\frac{n+2}{n}$.  Hence by \eqref{eq3.6},
\eqref{eq3.7}, \eqref{eq3.9}, and H\"older's inequality we get
 \begin{align}\label{eq3.15}
  \notag \iint_{(\R^n \times\R)\backslash E_\rho (x,t)}&\Phi(x-y,t-s)^{\frac{n+2-\alpha}{n}}f(y,s)\,dy\,ds\\
  \notag &\leq\| f\|_{L^p (\R^n \times\R)}\left(\iint_{E_\infty (x,t)\backslash E_\rho (x,t)}\Phi(x-y,t-s)^{\frac{n+2-\alpha}{n}q}\,dy\,ds\right)^{1/q}\\
  &\leq C\left(\frac{1}{\rho}\right)^{\frac{n+2}{p}-\alpha}\| f\|_{L^p (\R^n \times\R)}
 \end{align}
 where $C=C(n,\alpha,p)$ is a positive constant.
 
 Taking 
\[\rho=\left(\frac{\| f\|_p}{Mf(x,t)}\right)^{p/(n+2)}\] 
and adding \eqref{eq3.14} and \eqref{eq3.15} yields \eqref{eq3.12}.
\end{proof}

The following theorem is is the heat ball analog of the strong Hardy-Littlewood inequality for the maximal function \eqref{eq3.13}.

\begin{thm}\label{thm3.3}
 Let $f\in L^p (\R^n \times\R)$ be a nonnegative function where $p\in(1,\infty]$ and $n\geq 1$.  Then
 \begin{equation}\label{eq3.16}
  \| Mf\|_{L^p (\R^n \times\R)}\leq C\| f\|_{L^p (\R^n \times\R)}
 \end{equation}
 where $C=C(n,p)$ is a positive constant and $Mf$ is given by \eqref{eq3.13}.
\end{thm}

\begin{proof}
The theorem is trivially true if $p=\infty$.  Hence we can assume
  $1<p<\infty$.  Let $r_0$ be as in Lemma \ref{lem3.1}.  By
  \eqref{eq3.8} and \eqref{eq3.9} we have 
\[
\frac{|E_r (x,t)|}{|Q_{r_0
      r}(x,t)|}=\frac{r^{n+2}|E_1 (0,0)|}{(r_0 r)^{n+2}|Q_1
    (0,0)|}=C(n).
\]  
Thus by Lemma \ref{lem3.1},
 \begin{align*}
  Mf(x,t)&=\sup_{r>0}\frac{1}{|E_r (x,t)|}\iint_{E_r (x,t)}f(y,s)\,dy\,ds\\
  &\leq\sup_{r>0}\frac{1}{C(n)|Q_{r_0 r}(x,t)|}\iint_{Q_{r_0 r}(x,t)}f(y,s)\,dy\,ds\\
  &=\sup_{r>0}\frac{1}{C(n)|Q_r (x,t)|}\iint_{Q_r (x,t)}f(y,s)\,dy\,ds\\
  &=\frac{1}{C(n)}\widehat{M}f(x,t)
 \end{align*}
 where 
 \begin{equation}\label{eq3.17}
  \widehat{M}f(x,t)=\sup_{r>0}\frac{1}{|Q_r (x,t)|}\iint_{Q_r (x,t)}f(y,s)\,dy\,ds.
 \end{equation}
 Hence to complete the proof, it suffices to prove \eqref{eq3.16} with
 $Mf$ replaced with $\widehat{M}f$.  To do that we need the following
 $d$-ball analog of the weak Hardy-Littlewood inequality for the
 maximal function \eqref{eq3.17}.  By a $d$-ball we mean a ball in the
 metric space $(\R^n \times\R, d)$.

\begin{pro}\label{pro3.1}
 Let $g\in L^1 (\R^n \times\R)$ be a nonnegative function where $n\geq 1$.  Then 
 $$|\{\widehat{M}g>\lambda\}|
<\frac{5^{n+2}}{\lambda}\| g\|_{L^1 (\R^n \times\R)}\quad
\text {for all }\lambda>0.$$
\end{pro}

\begin{proof}
 Let $\lambda>0$ be fixed.  For each
 $(x,t)\in\{\widehat{M}g>\lambda\}$ there exists, by the definition of
 $\widehat{M}g$, $r(x,t)>0$ such that 
 $$\iint_{Q_{r(x,t)}(x,t)}g(y,s)\,dy\,ds>\lambda|Q_{r(x,t)}(x,t)|.$$
 Since $g\in L^1 (\R^n \times\R)$, the radii $r(x,t)$ of the balls 
$Q_{r(x,t)}(x,t)$, $(x,t)\in\{\widehat{M}g>\lambda\}$, are bounded.  
Thus by the Vitali covering lemma we can find among these balls a sequence
 $\{Q_j\}$ of pairwise disjoint balls such that 
 $$\bigcup^{\infty}_{j=1}5Q_j \supset\{\widehat{M}g>\lambda\}.$$
 Hence
 $$|\{\widehat{M}g>\lambda\}|\leq\sum^{\infty}_{j=1}|5Q_j |=\sum^{\infty}_{j=1}5^{n+2}|Q_j |<\frac{5^{n+2}}{\lambda}\| g\|_{L^1}.$$
\end{proof}

 Returning to the proof of Theorem \ref{thm3.3}, for each $\lambda>0$, we define $h_\lambda (x,t)=f(x,t)$ if $f(x,t)>\lambda/2$ and $0$ otherwise.  Since 
 $f\in L^p (\R^n \times\R),\,h_\lambda \in L^1 (\R^n \times\R)$.  Also,
 it is easy to check that 
\[
\{\widehat{M}f>\lambda\}\subset\{\widehat{M}h_\lambda >\lambda/2\}.
\]  
Thus, by Proposition \ref{pro3.1},
\begin{align}\label{eq3.18}
 \notag |\{\widehat{M}f>\lambda\}|&\leq|\{\widehat{M}h_\lambda >\lambda/2\}|\\
 \notag &\leq\frac{5^{n+2}}{\lambda/2}\| h_\lambda \|_{L^1 (\R^n \times\R)}\\
 &=\frac{2(5^{n+2})}{\lambda}\iint_{\{f>\lambda/2\}}f(y,s)\,dy\,ds.
\end{align}
Hence, by Lemma \ref{lem3.2} with $b=0,a=p-1,\alpha=1,m=n+1$, and $g=\widehat{M}f$ we have
\begin{align*}
 \|\widehat{M}f\|^{p}_p&=p\int^{\infty}_{0}
\lambda^{p-1}|\{\widehat{M}f>\lambda\}|\,d\lambda\\
 &\leq 2p5^{n+2}\int^{\infty}_{0}\lambda^{p-2}
\left(\iint_{\{f>\lambda/2\}}f(y,s)\,dy\,ds\right)d\lambda\\
 &=C(n,p)\| f\|^{p}_p
\end{align*}
where the last equation follows from Lemma \ref{lem3.2} with $b=1$,
$a=p-2$, $\alpha=2$, $m=n+1$, and $g=f$.
\end{proof}

The Sobolev inequality for heat potentials is given in the following theorem.

\begin{thm}\label{thm3.4}
 Suppose $0<\alpha<n+2$ and $1<p<\frac{n+2}{\alpha}$ are constants and
 $f:\R^n \times\R\to\R$ is a nonnegative measurable function.  Let 
$$q=\frac{(n+2)p}{n+2-\alpha p}.$$ 
 Then
 $$\| J_\alpha *f\|_{L^q (\R^n \times\R)}\leq C\| f\|_{L^p (\R^n
   \times\R)}$$ 
where $C=C(n,p,\alpha)$ is a positive constant.
\end{thm}
\begin{proof}
 By Theorems \ref{thm3.2} and \ref{thm3.3} we have 
 \begin{align*}
  \| J_\alpha *f\|_q &\leq C\| f\|^{\frac{\alpha p}{n+2}}_{p}\|(Mf)^{1-\frac{\alpha p}{n+2}}\|_q \\
  &=C\| f\|^{\frac{\alpha p}{n+2}}_{p}\| Mf\|^{\frac{n+2-\alpha p}{n+2}}_{p}\\
  &\leq C\| f\|_p.
 \end{align*}
\end{proof}

We are now able to prove Theorem \ref{nonlinear}.

\begin{proof}[Proof of Theorem \ref{nonlinear}] 
Let $g=(J_\beta\ast f)^\sigma$. By Theorem
  \ref{thm3.2}, we have
\begin{equation}\label{nonlin1}
\|J_\alpha \ast g\|_\infty\leq C \|g\|_p^{\frac{\alpha p}{n+2}}
\|g\|_\infty^{1-\frac{\alpha p}{n+2}}\quad\mbox{ for } 1\leq p<\frac{n+2}{\alpha}
\end{equation}
and
\begin{equation}\label{newone}
\|J_\beta \ast f\|_\infty\leq C \|f\|_r^{\frac{\beta r}{n+2}}
\|f\|_\infty^{1-\frac{\beta r}{n+2}}
\end{equation}
because 
\[
1\le r<\frac{(n+2)\sigma}{\alpha+\beta\sigma}=\frac{n+2}{\beta}
\frac{\beta\sigma}{\alpha+\beta\sigma}<\frac{n+2}{\beta}.
\]
Estimate \eqref{newone} implies
\begin{equation}\label{nonlin3}
\|g\|_\infty= \| J_\beta \ast f\|_\infty^\sigma 
\leq C \|f\|_r^{\frac{\sigma \beta r}{n+2}}
\|f\|_\infty^{\sigma-\frac{\sigma\beta r}{n+2}}.
\end{equation}
It follows from \eqref{assumptions} that there exist $s\in(r,(n+2)/\beta)$ and
$p\in(1,(n+2)/\alpha)$ such that
\begin{equation}\label{nonlin2}
p\sigma=\frac{(n+2)s}{n+2-\beta s}.
\end{equation}
By Theorem \ref{thm3.4} we have
\begin{equation}\label{nonlin4}
\|g\|_p=\|J_\beta \ast f\|^\sigma_{p\sigma} \leq C \|f\|_s^\sigma.
\end{equation}
We now use \eqref{nonlin4} and \eqref{nonlin3} in \eqref{nonlin1} to obtain
\begin{equation}\label{nonlin5}
\| J_\alpha\ast((J_\beta\ast f)^\sigma)\|_\infty=\|J_\alpha \ast g\|_\infty
\leq C \|f\|_s^{\frac{\alpha p\sigma}{n+2}} \|f\|_r^{\frac{\sigma
    \beta r}{n+2}(1-\frac{\alpha p}{n+2})}
\|f\|_\infty^{(\sigma-\frac{\sigma\beta r}{n+2})(1-\frac{\alpha p}{n+2})}.
\end{equation}
Finally, using the estimate
$$
\|f\|_s\leq \|f\|_r^{\frac{r}{s}}\|f\|_\infty^{\frac{s-r}{s}}
$$
in \eqref{nonlin5} gives
$$
\begin{aligned}
\|J_\alpha\ast((J_\beta\ast f)^\sigma)\|_\infty&\leq C
\|f\|_r^{\frac{\alpha p\sigma r}{(n+2)s}+ \frac{\sigma
    \beta r}{n+2}(1-\frac{\alpha p}{n+2})}
\|f\|_\infty^{\frac{\alpha p\sigma}{n+2}\frac{s-r}{s}+(\sigma-\frac{\sigma\beta r}{n+2})(1-\frac{\alpha p}{n+2})}\\
&= C \|f\|_r^{\frac{(\alpha+\beta\sigma)r}{n+2}} 
\|f\|_\infty^{\frac{\sigma(n+2-\beta r)-\alpha r}{n+2}}
\end{aligned}
$$
by \eqref{nonlin2}.
\end{proof}

Let $\Omega=\R^n\times (a,b)$ where $n\ge 1$ and $a<b$. The following
theorem gives estimates for the heat potential
\[
(V_\alpha f)(x,t)=\iint_{\Omega}
  \Phi(x-y,t-s)^{\frac{n+2-\alpha}{n}}f(y,s)\,dy\,ds, 
\]
where $\Phi$ is given by \eqref{eq1.5} and $\alpha\in(0,n+2)$.

\begin{thm}\label{HeatPotential}
Let $p,q\in[1,\infty]$,  $\alpha$, and $\delta$ satisfy 
\begin{equation}\label{heatA1}
0\le \delta =\frac{1}{p}-\frac{1}{q}<\frac{\alpha}{n+2}<1.
\end{equation}
Then $V_\alpha$ maps $L^p(\Omega)$ continuously into 
$L^q(\Omega)$ and for $f\in L^p(\Omega)$ we have
\begin{equation}\label{heatA2}
\|V_\alpha  f\|_{L^q(\Omega)} \le M\|f\|_{L^p(\Omega)},
\end{equation}
where 
\begin{equation}\label{heatA3}
M=C(b-a)^{(\alpha-(n+2)\delta)/2}\quad\text{for some constant }
C=C(n,\alpha,\delta)>0. 
\end{equation}
\end{thm}

Theorem \ref{HeatPotential} is weaker than Theorem \ref{thm3.4} in
that the second inequality in \eqref{heatA1} cannot be replaced with
equality. However it is stronger in that the cases $p=1$ and
$q=\infty$ are allowed. These cases will be needed in Section
\ref{sec5} to prove Theorem \ref{thm2.2}.

\begin{proof} This proof is a modification of the proof of Lemma 7.12
  in \cite{GT} dealing with Riesz potentials.  Let
  $\beta=\frac{n+2}{n}(1-\frac{\alpha}{n+2})$ and
  $r=\frac{1}{1-\delta}$. Then by \eqref{heatA1}
\[
1-\frac{n(\beta r-1)}{2}=\frac{n+2}{2}\frac{\frac{\alpha}{n+2}-\delta}{1-\delta}>0
\]
and for $s<t$ we have
\begin{align*}
\int_{\R^n}\Phi(x-y,t-s)^{\beta r}\,dy
&=\int_{\R^n}\Phi(x-y,t-s)^{\beta r}\,dx\\
&=\frac{1}{(4\pi)^{(\beta r-1)n/2}(\beta r)^{n/2}(t-s)^{(\beta r-1)n/2}}.
\end{align*}
Hence, letting $X=(x,t)$, $Y=(y,s)$, and $h=\Phi^\beta$ we have
\begin{equation}\label{heatA4}
\|h(X-\cdot)\|_{L^r(\Omega)}\le M \qquad \text{for all } X\in\Omega
\end{equation}
and
\begin{equation}\label{heatA5}
\|h(\cdot-Y)\|_{L^r(\Omega)}\le M \qquad \text{for all } Y\in\Omega
\end{equation}
where $M$ is given by \eqref{heatA3}.

Since
\[
\frac{r}{q}+r\left(1-\frac{1}{p}\right)
=r\left(1+\frac{1}{q}-\frac{1}{p}\right)=r(1-\delta)=1
\]
and
\[
\frac{p}{q}+p\delta=p\left(\frac{1}{q}+\delta\right)=1
\]
we have
\[
h(X-Y)|f(Y)|=h(X-Y)^{r/q}|f(Y)|^{p/q}h(X-Y)^{r(1-1/p)}|f(Y)|^{p\delta}.
\]
Thus by H\"older's inequality and the fact that
\[
\frac{1}{q}+\left(1-\frac{1}{p}\right)+\delta=1
\]
we have 
\begin{align*}
|V_\alpha  f(X)|&\le \int_\Omega h(X-Y)|f(Y)|\,dY\\
&\le \left(\int_\Omega h(X-Y)^r|f(Y)|^p\,dY\right)^{1/q}
     \left(\int_\Omega h(X-Y)^r\,dY\right)^{1-1/p}
     \left(\int_\Omega |f(Y)|^p\,dY\right)^\delta.
\end{align*}
So by \eqref{heatA4}
\[
\|V_\alpha  f\|_{L^q(\Omega)}\le M^{r(1-1/p)}\left(\int_\Omega
  |f(Y)|^p\,dY\right)^\delta J
\]
where 
\begin{align*}
J:&=\left(\int_\Omega\int_\Omega h(X-Y)^r|f(Y)|^p\,dY\,dX\right)^{1/q}\\
&=\left(\int_\Omega|f(Y)|^p\left(\int_\Omega
    h(X-Y)^r\,dX\right)\,dY\right)^{1/q}\\
&\le M^{r/q}\left(\int_\Omega |f(Y)|^p\,dY\right)^{1/q}
\end{align*}
by \eqref{heatA5}. Hence \eqref{heatA2} follows from \eqref{heatA1}.
\end{proof}

\section{Preliminary lemmas}\label{sec4}
In this section we provide some lemmas needed for the proofs of our
results in Section \ref{sec2}. 

\begin{lem}\label{lem4.1}
Suppose $u$ is a $C^{2,1}$ nonnegative solution of
\begin{equation}\label{eq4.1}
 Hu\ge 0\quad \text{in}\quad B_{\sqrt{4R}}(0)\times (0,4R) 
\subset \R^n \times \R,\ n\ge 1,
\end{equation}
where $Hu = u_t-\Delta u$ is the heat operator and $R$ is a positive
constant. Then
\begin{equation}\label{eq4.2}
 u,Hu \in L^1(B_{\sqrt{2R}}(0) \times (0,2R))
\end{equation}
and there exist a finite positive Borel measure $\mu$ on
$B_{\sqrt{2R}}(0)$ and a bounded function\\
$h\in C^{2,1}(B_{\sqrt{R}}(0) \times (-R,R))$ satisfying
\begin{alignat}{2}
 \label{eq4.3}
Hh &= 0  \quad \text{in} &\quad &B_{\sqrt{R}}(0) \times (-R,R)\\
\label{eq4.4}
h &= 0 \quad \text{in} &\quad &B_{\sqrt{R}}(0)\times (-R,0] 
\end{alignat}
such that
\begin{equation}\label{eq4.5}
 u = N +v+h\quad \text{in}\quad B_{\sqrt{R}}(0)\times (0,R)
\end{equation}
where
\begin{align}\label{eq4.6}
 N(x,t) &:= \int^{2R}_0 \int_{|y|<\sqrt{2R}} \Phi(x-y,t-s) Hu(y,s)\,dy\,ds,\\
\label{eq4.7}
v(x,t) &:= \int_{|y|<\sqrt{2R}} \Phi(x-y,t)\,d\mu(y),
\end{align}
and $\Phi$ is the heat kernel \eqref{eq1.5}.
\end{lem}

\begin{proof}
When $R=1$, Lemma \ref{lem4.1} was proved in \cite{T2011}. The proof of
Lemma \ref{lem4.1} for $R$ any positive constant is obtained by
scaling the $R=1$ case.
\end{proof}

Watson \cite{W1976} provided another representation formula for
distributional solutions of \eqref{eq4.1} in terms of integral
potentials involving the Green function of the heat operator. See also
Hirata \cite{H2014}.

\begin{rem}\label{rem2}
Under the assumptions of Lemma \ref{lem4.1} we have 
\begin{equation}\label{eq4.8}
(4\pi t)^{n/2} v(x,t) \le \int_{|y|<\sqrt{2R}} d\mu(y)<\infty
\quad \text{for } (x,t)\in \R^n \times (0,\infty).
\end{equation}
Thus by \eqref{eq4.5} we see that 
\[
u(x,t)\le C\left(\left(\frac{1}{\sqrt{t}}\right)^n + N(x,t)\right) 
\quad\text{for }(x,t)\in B_{\sqrt{R}}(0)\times (0,R).
\]
\end{rem}

To prove our results in Section \ref{sec2}, it will be convenient to
use instead of the sets $P_r(x,t)$ and $E_r(x,t)$ the sets  
${\cal P}_r(x,t)$ and ${\cal E}_r(x,t)$ defined by 
\begin{equation}
\begin{aligned}\label{eq4.9}
{\cal P}_r(x,t)&=P_{\sqrt{r}}(x,t)\\
{\cal E}_r(x,t)&=E_{\sqrt{r}}(x,t)
\end{aligned}
\quad\text{for }(x,t)\in\R^n\times\R\quad\text{ and }r>0.
\end{equation}
It follows from
\eqref{eq3.8} and \eqref{eq3.9} that 
\begin{equation}\label{eq4.10}
|{\cal P}_r(x,t)|=r^{\frac{n+2}{2}}|{\cal P}_1(0,0)|
\end{equation}
\begin{equation}\label{eq4.11}
|{\cal E}_r(x,t)|=r^{\frac{n+2}{2}}|{\cal E}_1(0,0)|.
\end{equation}
Also, by Lemma \ref{lem3.1},
\begin{equation}\label{eq4.12}
{\cal E}_r(x,t)\subset {\cal P}_{r_0^2r}(x,t)
\end{equation}
where $r_0=r_0(n)$ is as in Lemma \ref{lem3.1}.

\begin{lem}\label{lem4.2}
  Suppose $(x_0,t_0)\in\R^n\times\R$ and $r>0$. If
  \[
(x,t)\in\overline{{\cal P}_r(x_0,t_0)} \quad \text{and}\quad
  (y,s)\in(\R^n\times\R)\setminus\overline{{\cal P}_{2r}(x_0,t_0)}
\] then
\[
\Phi(x-y,t-s)\le\frac{C(n)}{r^{n/2}}.
\] 
\end{lem}

\begin{proof} {\bf Case I}. Suppose $t_0-2r\le s<t$. Then $|x-y|\ge
  (\sqrt{2}-1) \sqrt{r}$ and hence 
\begin{align*}
\Phi(x-y,t-s)&\le\frac{e^{-\frac{(\sqrt{2}-1)^2r}{4(t-s)}}}{(4\pi
  (t-s))^{n/2}}\le\sup_{\tau>0}\frac{e^{-\frac{(\sqrt{2}-1)^2r}{4\tau}}}{(4\pi
  \tau)^{n/2}}\\
             &=\sup_{\zeta>0}\frac{e^{-(\sqrt{2}-1)^2\zeta}}{(\pi r/\zeta)^{n/2}}
=\frac{C(n)}{r^{n/2}}.
\end{align*}
\noindent {\bf Case II}. Suppose $s<t_0-2r$. Then $t-s\ge r$ and hence 
\[
\Phi(x-y,t-s)\le\frac{1}{(4\pi r)^{n/2}}=\frac{C(n)}{r^{n/2}}.
\]
\noindent {\bf Case III}. Suppose $s\ge t$. Then $\Phi(x-y,t-s)=0$.
\end{proof}

\begin{lem}\label{lem4.3}
 Suppose $K$ is a compact subset of an open set $\Omega\subset\mathbb{R}^n ,\,n\geq 1$, and $u(x,t)$ is a $C^{2,1}$ nonnegative solution of
 \begin{equation}\label{eq4.13}
  Hu\geq 0\quad\text{in }\Omega\times(0,1).
 \end{equation}
 Let $\{(x_j ,t_j )\}\subset K\times(0,1)$ be a sequence such that
 \begin{equation}\label{eq4.14}
  t_j \to 0 \quad\text{as }j\to\infty.
 \end{equation}
 Then for some subsequence of $\{(x_j ,t_j )\}$, which we denote again by $\{(x_j ,t_j )\}$, we have 
 \begin{equation}\label{eq4.15}
  {\cal P}_{t_j}(x_j ,t_j )\subset \Omega\times(0,1),
 \end{equation}
 \begin{equation}\label{eq4.16}
  \iint_{{\cal P}_{t_j}(x_j ,t_j )}Hu(x,t)\,dx\,dt\to 0\quad\text{as }j\to\infty,
 \end{equation}
 and, for all $a\geq 1$,
 \begin{equation}\label{eq4.17}
  u(x,t)\leq C\left[\left(\frac{1}{\sqrt{t_j}}\right)^n +\iint_{{\cal P}_{t_j
        /a}(x_j ,t_j )}\Phi(x-y,t-s)Hu(y,s)\,dy\,ds\right]\quad\text{ for }(x,t)
\in\overline{{\cal P}_{t_j/2a}(x_j ,t_j )}
 \end{equation}
 where $C>0$ does not depend on $(x,t)$ or $j$ (but may depend on $a$).
\end{lem}

\begin{proof}
  By taking a subsequence of $\{(x_j ,t_j )\}$ we can assume there 
exists $x_0 \in K$ such that 
  $x_j \to x_0$ as $j\to\infty,$ and, for some $\varepsilon>0$,
  ${\cal P}_{4\varepsilon}(x_0 ,4\varepsilon)\subset\Omega\times(0,1)$ and
 \begin{equation}\label{eq4.18}
  \overline{{\cal P}_{t_j}(x_j ,t_j )}\subset{\cal P}_\varepsilon (x_0 ,\varepsilon)\quad\text{ for }j=1,2,\ldots.
 \end{equation}
 Thus \eqref{eq4.15} holds.
 
 By \eqref{eq4.13}, Lemma \ref{lem4.1}, and Remark
 \ref{rem2}, we have
 \begin{equation}\label{eq4.19}
  \iint_{{\cal P}_{2\varepsilon}(x_0,2\varepsilon)}Hu(x,t)\,dx\,dt<\infty
 \end{equation}
 and, for $(x,t)\in{\cal P}_\varepsilon (x_0 ,\varepsilon)$,
 \begin{equation}\label{eq4.20}
  u(x,t)\leq C\left[\left(\frac{1}{\sqrt{t}}\right)^n +\iint_{{\cal P}_{2\varepsilon}(x_0 ,2\varepsilon)}\Phi(x-y,t-s)Hu(y,s)\,dy\,ds\right] 
 \end{equation}
 where $C>0$ does not depend on $(x,t)$.  However, for
 $$(x,t)\in\overline{{\cal P}_{t_j/2a}(x_j ,t_j )}\quad\text{ and
 }\quad(y,s)\in{\cal P}_{2\varepsilon}(x_0
 ,2\varepsilon)\backslash\overline{{\cal P}_{t_j/a}(x_j ,t_j )}$$ 
we have by Lemma \ref{lem4.2} that
 $$\Phi(x-y,t-s)\leq C(n)\left(\frac{a}{t_j}\right)^{n/2}$$ 
and thus by \eqref{eq4.19} we find that
 $$\iint_{{\cal P}_{2\varepsilon}(x_0 ,2\varepsilon)\backslash{\cal
     P}_{t_j/a}(x_j ,t_j
   )}\Phi(x-y,t-y)Hu(y,s)\,dy\,ds<\frac{C}{t^{n/2}_{j}}\quad\text{ for }(x,t)
\in\overline{{\cal P}_{t_j/2a}(x_j ,t_j )}.$$
Inequality \eqref{eq4.17} therefore follows from \eqref{eq4.20}.  
Finally, \eqref{eq4.18} and \eqref{eq4.19} imply \eqref{eq4.16}. 
\end{proof}

\begin{lem}\label{lem4.4}
  Suppose $\Omega_2 \subset\subset\Omega_1 \subset\subset\Omega_0$ are
  nonempty open subsets of $\mathbb{R}^n ,n\geq 1$.  Let $u(x,t)$ be a
  $C^{2,1}$ nonnegative solution of
 $$Hu\geq 0\quad\text{ in }\Omega_0 \times(0,1)$$ 
satisfying
 \begin{equation}\label{eq4.21}
 \max_{x\in\overline{\Omega}_1}Hu(x,t)
=O\left(\left(\frac{1}{\sqrt{t}}\right)^\gamma \right)\quad\text{ as }t\to 0^+  
 \end{equation}
 where $\gamma$ is a real constant.  Then
 \begin{equation}\label{eq4.22}
 \max_{x\in\overline{\Omega}_2}u(x,t)=O\left(\left(\frac{1}{\sqrt{t}}\right)^n \right)+o\left(\left(\frac{1}{\sqrt{t}}\right)^{\gamma\frac{n}{n+2}}\right)\quad\text{ as }t\to 0^+ .
 \end{equation}
 If, in addition, $\gamma>n+2$ and $v(x,t)$ is a $C^{2,1}$ nonnegative solution of
 \begin{equation}\label{eq4.23}
  0\leq Hv\leq \left(u+\left(\frac{1}{\sqrt{t}}\right)^n \right)^\sigma \quad\text{ in }\Omega_0 \times(0,1) 
 \end{equation}
 where $\sigma>\frac{2}{n}$ then
 \begin{equation}\label{eq4.24}
  \max_{x\in\overline{\Omega}_2}v(x,t)
=O\left(\left(\frac{1}{\sqrt{t}}\right)^n \right)
+o\left(\left(\frac{1}{\sqrt{t}}\right)^{\gamma\frac{n\sigma-2}{n+2}}\right)
\quad\text{as }t\to 0^+.
 \end{equation}
\end{lem}

\begin{proof}
  For the proof of \eqref{eq4.22} we can assume $\gamma\geq n+2$
  because increasing $\gamma$ to $n+2$ weakens condition
  \eqref{eq4.21} and does not change \eqref{eq4.22}.
 
 Suppose for contradiction that \eqref{eq4.22} is false.  Then there exists a sequence $\{(x_j ,t_j )\}\subset\overline{\Omega}_2 \times(0,1)$ such that $t_j \to 0$ as $j\to\infty$ and either
 \begin{equation}\label{eq4.25}
  \lim_{j\to\infty}\sqrt{t_j}^n u(x_j ,t_j )=\infty\quad\text{if }n\geq\gamma\frac{n}{n+2}
 \end{equation}
 or
 \begin{equation}\label{eq4.26}
  \liminf_{j\to\infty}\sqrt{t_j}^{\gamma\frac{n}{n+2}}u(x_j ,t_j )>0\quad\text{if }n<\gamma\frac{n}{n+2}.
 \end{equation}
 By taking a subsequence, we have by Lemma \ref{lem4.3} with $\Omega=\Omega_1$, $K=\overline{\Omega}_2$, and $a=2$ applied to the function $u$ that the sequence $\{(x_j ,t_j )\}$ satisfies
 \begin{equation}\label{eq4.27}
  {\cal P}_{t_j}(x_j ,t_j )\subset\Omega_1 \times(0,1)
 \end{equation}
 and the function $u$ satisfies \eqref{eq4.16} and \eqref{eq4.17}.
 
 By \eqref{eq4.21} and \eqref{eq4.27} we have
 \begin{equation}\label{eq4.28}
  Hu(x,t)\leq\frac{A}{\sqrt{t_j}^\gamma}\quad\text{ for }
(x,t)\in{\cal P}_{t_j /2}(x_j ,t_j )
 \end{equation}
 where $A$ is a positive constant which does not depend on $(x,t)$ or $j$.
 
 Define $r_j \geq 0$ by
 \begin{equation}\label{eq4.29}
  \iint_{{\cal E}_{r_j}(x_j ,t_j )}\frac{A}{\sqrt{t_j}^\gamma}\,dx\,dt=\iint_{{\cal P}_{t_j /2}(x_j ,t_j )}Hu(x,t)\,dx\,dt\to 0\quad\text{ as }j\to\infty
 \end{equation}
 by \eqref{eq4.16}.  Then by \eqref{eq4.11} we have
\begin{equation}\label{eq4.30}
  r_j =o\left(t^{\frac{\gamma}{n+2}}_{j}\right)<<t_j \quad\text{ as }j\to\infty
 \end{equation}
 because $\gamma\geq n+2$.  
Hence by \eqref{eq4.12}, 
\[
{\cal E}_{r_j}(x_j,t_j)\subset {\cal P}_{t_j/2}(x_j,t_j)
\quad\text{for large }j.
\]
Thus by \eqref{eq4.9}, \eqref{eq4.28} and \eqref{eq4.29} we have for
large $j$ that
\begin{align*}
\iint_{{\cal E}_{r_j}(x_j,t_j)}
&\Phi(x_j-y,t_j-s)\left(\frac{A}{\sqrt{t_j}^\gamma}-Hu(y,s)\right)dy\,ds\\
&\ge \frac{1}{r_j^{n/2}}\iint_{{\cal E}_{r_j}(x_j,t_j)}
\left(\frac{A}{\sqrt{t_j}^\gamma}-Hu(y,s)\right)\,dy\,ds\\
&=\frac{1}{r_j^{n/2}}\iint_{{\cal P}_{t_j/2}(x_j,t_j)\setminus{\cal E}_{r_j}(x_j,t_j)}
Hu(y,s)\,dy\,ds\\
&\ge \iint_{{\cal P}_{t_j/2}(x_j,t_j)\setminus{\cal E}_{r_j}(x_j,t_j)}
\Phi(x_j-y,t_j-s)Hu(y,s)\,dy\,ds.
\end{align*}
So for large $j$ we have
\begin{align*}
\iint_{{\cal P}_{t_j/2}(x_j,t_j)}\Phi(x_j-y,t_j-s)Hu(y,s)\,dy\,ds
&\le \frac{A}{\sqrt{t_j}^\gamma}\iint_{{\cal E}_{r_j}(x_j,t_j)}
\Phi(x_j-y,t_j-s)\,dy\,ds\\
&\le\frac{Ar_0^2r_j}{\sqrt{t_j}^\gamma}
\end{align*}
by \eqref{eq4.12} and the fact that $\int_{\R^n}\Phi(x_j-y,t_j-s)\,dy\,ds=1$
for $s<t_j$.
Hence by \eqref{eq4.17} and \eqref{eq4.30}  we find that
\begin{align*}
  u(x_j ,t_j )&\leq C\left[\left(\frac{1}{\sqrt{t_j}}\right)^n 
+\iint_{{\cal P}_{t_j /2}(x_j ,t_j )}\Phi(x_j -y,t_j -s)Hu(y,s)\,dy\,ds\right]\\
 &\leq C\left[\left(\frac{1}{\sqrt{t_j}}\right)^n +\frac{Ar_0^2r_j}{\sqrt{t_j}^\gamma}\right]\\
  &=C\left[\left(\frac{1}{\sqrt{t_j}}\right)^n +o\left(\left(\frac{1}{\sqrt{t_j}}\right)^{\gamma\frac{n}{n+2}}\right)\right]\quad\text{ as }j\to 0
 \end{align*}
 which contradicts \eqref{eq4.25}, \eqref{eq4.26} and thereby proves
 \eqref{eq4.22}.
 
 Suppose for contradiction that \eqref{eq4.24} is false.  Then there
 exists a sequence $\{(x_j ,t_j )\}\subset\overline{\Omega}_2
 \times(0,1)$ such that $t_j \to 0$ as $j\to\infty$ and either
 \begin{equation}\label{eq4.31}
  \lim_{j\to\infty}\sqrt{t_j}^n v(x_j ,t_j )=\infty\quad\text{ if }n\geq\gamma\frac{n\sigma-2}{n+2}
 \end{equation}
 or
 \begin{equation}\label{eq4.32}
  \liminf_{j\to\infty}\sqrt{t_j}^{\gamma\frac{n\sigma-2}{n+2}}v(x_j ,t_j )>0\quad\text{ if }n<\gamma\frac{n\sigma-2}{n+2}.
 \end{equation}
 By taking a subsequence, we have by Lemma \ref{lem4.3} with
 $\Omega=\Omega_1 $, $K=\overline{\Omega}_2$ , and $a=2$ applied to the
 function $u$ that the sequence $\{(x_j ,t_j )\}$ satisfies
 \eqref{eq4.27} and the function $u$ satisfies \eqref{eq4.16} and
 \eqref{eq4.17}.  Thus for $(x,t)\in\overline{{\cal P}_{t_j/4}(x_j
   ,t_j )}$ we have by \eqref{eq4.23} that
 $$Hv(x,t)\leq\left(u(x,t)+\left(\frac{1}{\sqrt{t}}\right)^n\right)^\sigma
 \leq
 C\left[\left(\frac{1}{\sqrt{t_j}}\right)^{n\sigma}+\left((N_{{\cal
         P}_{t_j/2}(x_j ,t_j )}(Hu))(x,t)\right)^\sigma \right]$$ 
where 
 $$(N_D f)(x,t):=\iint_D \Phi(x-y,t-s)f(y,s)\,dy\,ds.$$  
 Hence applying Lemma \ref{lem4.3} to $v$ with $\Omega=\Omega_1$,
 $K=\overline{\Omega}_2$, and $a=4$  we get
 \begin{align}\label{eq4.33}
   \notag v(x_j ,t_j )&\leq C\left[\left(\frac{1}{\sqrt{t_j}}\right)^n 
+\iint_{{\cal P}_{t_j/4}(x_j ,t_j )}\Phi(x_j -y,t_j -s)Hv(y,s)\,dy\,ds\right]\\
   &\leq C\left[\left(\frac{1}{\sqrt{t_j}}\right)^n
     +\left(\frac{1}{\sqrt{t_j}}\right)^{n\sigma-2}+(K_j (Hu))(x_j ,t_j )\right]
 \end{align}
 where
 $$K_j f=N_{{\cal P}_{t_j/2}(x_j ,t_j )}\left((N_{{\cal P}_{t_j/2}(x_j ,t_j )}f)^\sigma \right).$$
 Since $\sigma>\frac{2}{n}$ we find using \eqref{eq4.16} and
 \eqref{eq4.21} in Theorem \ref{nonlinear} (with $\alpha=\beta=2$ and
 $r=1$) that
 $$(K_j (Hu))(x_j ,t_j)
=o\left(\left(\frac{1}{\sqrt{t_j}}\right)^{\gamma\frac{n\sigma-2}{n+2}}\right)
\quad\text{as }j\to\infty.$$  Thus \eqref{eq4.33} contradicts 
(\ref{eq4.31}, \ref{eq4.32}).  This completes the proof of \eqref{eq4.24}.
\end{proof}

\begin{lem}\label{lem4.5}
 Suppose $u$ and $v$ are $C^{2,1}$ nonnegative solutions of the system 
\begin{equation}
 \begin{aligned}\label{eq4.34}
   &0\leq Hu\\
  &0\leq Hv\leq\left(u+\left(\frac{1}{\sqrt{t}}\right)^n \right)^\sigma 
 \end{aligned}
\quad\quad\text{ in }\Omega\times(0,1)
\end{equation}
 where $\Omega$ is a open subset of $\mathbb{R}^n ,\,n\geq 1$.  Let $K$ be a compact subset of $\Omega$.
 \begin{enumerate}
  \item[(i)] If $\sigma<2/n$ then 
  \begin{equation}\label{eq4.35}
   \max_{x\in K}v(x,t)=O\left(\left(\frac{1}{\sqrt{t}}\right)^n \right)\quad\text{ as }t\to0^+ .
  \end{equation}
  \item[(ii)] If 
  \begin{equation}\label{eq4.36}
   \lambda>\frac{n+2}{n}\quad\text{ and }\quad
\sigma<\frac{2}{n}+\frac{n+2}{n\lambda}
  \end{equation}
  and
  \begin{equation}\label{eq4.37}
   Hu\leq\left(v+\left(\frac{1}{\sqrt{t}}\right)^n \right)^\lambda \quad\text{ in }\Omega\times(0,1)
  \end{equation}
  then for some $\gamma>n+2$ we have
  \begin{equation}\label{eq4.38}
   \max_{x\in K}Hu(x,t)=O\left(\left(\frac{1}{\sqrt{t}}\right)^\gamma \right)\quad\text{ as }t\to0^+ .
  \end{equation}
 \end{enumerate}
\end{lem}

\begin{proof}
 We can assume for the proof of (i) (resp. (ii)) that
 \begin{equation}\label{eq4.39}
  0<\sigma<\frac{2}{n}
 \end{equation}
 \begin{equation}\label{eq4.40}
  \left(\text{resp. }\frac{2}{n}<\sigma<\frac{2}{n}
+\frac{n+2}{n\lambda}\right)
 \end{equation}
 because increasing $\sigma$ weakens the condition \eqref{eq4.34}$_2$ on $v$ but does not change the estimates \eqref{eq4.35} or \eqref{eq4.38}.
 
 Suppose for contradiction that (i) (resp. (ii)) is false.  Then there exists a sequence \\
 $\{(x_j ,t_j )\}\subset K\times(0,1)$ such that $t_j \to0$ as $j\to\infty$ and
 \begin{equation}\label{eq4.41}
  \sqrt{t_j}^n v(x_j ,t_j )\to\infty\quad\text{ as }j\to\infty
 \end{equation}
 \begin{equation}\label{eq4.42}
  \text{(resp. }\sqrt{t_j}^\gamma Hu(x_j ,t_j )\to\infty\quad\text{ as }j\to\infty
 \end{equation}
 for all $\gamma>n+2$).  To obtain a single sequence $\{(x_j ,t_j )\}$
 such that \eqref{eq4.42} holds for all $\gamma>n+2$, one uses a
 standard diagonalization argument.
 
 By taking a subsequence we have by Lemma \ref{lem4.3} that ${\cal P}_{t_j}(x_j ,t_j )\subset\Omega\times(0,1)$, 
 \begin{equation}\label{eq4.43}
  \iint_{{\cal P}_{t_j/2}(x_j ,t_j )}Hu(y,s)\,dy\,ds\to0
\quad\text{ and }\iint_{{\cal P}_{t_j/2}(x_j ,t_j )}Hv(y,s)\,dy\,ds\to0\quad\text{ as }j\to\infty,
 \end{equation}
 and for $(x,t)\in\overline{{\cal P}_{t_{j/4}}(x_j ,t_j )}$ we have
 \begin{equation}\label{eq4.44}
  u(x,t)\leq C\left[\left(\frac{1}{\sqrt{t_j}}\right)^n +\iint_{{\cal P}_{t_j/2}(x_j ,t_j )}\Phi(x-y,t-s)Hu(y,s)\,dy\,ds\right] 
 \end{equation}
 \begin{equation}\label{eq4.45}
  v(x,t)\leq C\left[\left(\frac{1}{\sqrt{t_j}}\right)^n +\iint_{{\cal P}_{t_j/2}(x_j ,t_j )}\Phi(x-y,t-s)Hv(y,s)\,dy\,ds\right] 
 \end{equation}
 where $C>0$ does not depend on $(x,t)$ or $j$.  
 
 Define $f_j ,g_j :{\cal P}_2 (0,2)\to[0,\infty)$ by 
 $$f_j (\eta,\zeta)=r^{\frac{n+2}{2}}_{j}Hu(x_j +\sqrt{r_j}\eta,t_j +r_j \zeta)$$
 $$g_j (\eta,\zeta)=r^{\frac{n+2}{2}}_{j}Hv(x_j +\sqrt{r_j}\eta,t_j +r_j \zeta)$$
 where $r_j =t_j /4$.  Making the change of variables
 $$x=x_j +\sqrt{r_j}\xi,\quad t=t_j +r_j \tau$$
 $$y=x_j +\sqrt{r_j}\eta,\quad s=t_j +r_j\zeta$$
 in \eqref{eq1.5}, \eqref{eq4.43}, \eqref{eq4.44}, and \eqref{eq4.45} we get
 $$\Phi(x-y,t-s)=r^{-n/2}_{j}\Phi(\xi-\eta,\tau-\zeta),$$
 \begin{equation}\label{eq4.46}
  \iint_{{\cal P}_2(0,0)}f_j (\eta,\zeta)\,d\eta\,d\zeta\to0
\quad\text{ and }\quad\iint_{{\cal P}_2 (0,0)}g_j
(\eta,\zeta)\,d\eta\,d\zeta\to0
\quad\text{as }j\to\infty,
 \end{equation}
 and
 \begin{equation}\label{eq4.47}
  u(x_j +\sqrt{r_j}\xi,t_j +r_j \tau)\leq\frac{C}{r^{n/2}_{j}}\left[1+(N_2 f_j )(\xi,\tau)\right]\quad\text{ for }(\xi,\tau)\in{\cal P}_1 (0,0)
 \end{equation}
 \begin{equation}\label{eq4.48}
  v(x_j +\sqrt{r_j}\xi,t_j +r_j \tau)\leq\frac{C}{r^{n/2}_{j}}\left[1+N_2 g_j (\xi,\tau)\right]\quad\text{ for }(\xi,\tau)\in{\cal P}_1 (0,0)
 \end{equation}
 where
 $$(N_R f)(\xi,\tau):=\iint_{{\cal P}_R (0,0)}\Phi(\xi-\eta,\tau-\zeta)f(\eta,\zeta)\,d\eta\,d\zeta.$$
 
 We now prove part (i).  Define $\varepsilon\in(0,1)$ and $\gamma>0$
 by $\sigma=\frac{2}{n}(1-\varepsilon)^2$ and
 $\gamma=\frac{n+2}{n}(1-\varepsilon)$.  It follows from
 \eqref{eq4.46} and Theorem \ref{HeatPotential} with $p=1$ and
 $\alpha=2$ that $N_2 f_j \to0$ in
 $L^\gamma ({\cal P}_2 (0,0))$ and hence
 $$(N_2 f_j )^\sigma \to0\quad\text{ in }L^{\frac{n+2}{2(1-\varepsilon)}}({\cal P}_2 (0,0)).$$  Thus by H\"older's inequality
 \begin{equation}\label{eq4.49}
  \iint_{{\cal P}_1 (0,0)}\Phi^* (N_2 f_j )^\sigma \,d\eta\,d\zeta\leq\|\Phi^* \|_{\frac{n+2}{n+2\varepsilon}}\|(N_2 f_j )^\sigma \|_{\frac{n+2}{2(1-\varepsilon)}}\to0\quad\text{ as }j\to\infty
 \end{equation}
 where $\Phi^* (\eta,\zeta)=\Phi(\eta,-\zeta)$.  By \eqref{eq4.48} and
 \eqref{eq4.46} we have
 \begin{align}\label{eq4.50}
  \notag v(x_j ,t_j )&\leq\frac{C}{\sqrt{t_j}^n}\left(1+\iint_{{\cal P}_2 (0,0)}\Phi^* g_j \,d\xi \,d\tau\right)\\
  &\leq\frac{C}{\sqrt{t_j}^n}\left(1+\iint_{{\cal P}_1 (0,0)}\Phi^* g_j \,d\xi \,d\tau\right)
 \end{align}
 and for $(\xi,\tau)\in{\cal P}_1 (0,0)$ it follows from
 \eqref{eq4.34}$_2$ and \eqref{eq4.47} that
 \begin{align}\label{eq4.51}
  \notag g_j (\xi,\tau)&=r^{\frac{n+2}{2}}_{j}(Hv)(x,t)\\
  \notag &\leq r^{\frac{n+2}{2}}_{j}\left(u(x,t)
+\left(\frac{1}{\sqrt{t_j}}\right)^n \right)^\sigma \\
  &\leq C(\sqrt{t_j})^{n+2-n\sigma}(1+N_2 f_j (\xi,\tau))^\sigma .
 \end{align}
 Substituting \eqref{eq4.51} in \eqref{eq4.50} and using \eqref{eq4.49} we get $v(x_j ,t_j )\leq C\frac{1}{\sqrt{t_j}^n}$ which contradicts \eqref{eq4.41} and thereby completes the proof of part (i).
 
 We next prove part (ii). It follows from \eqref{eq4.46}, \eqref{eq4.47}, \eqref{eq4.48} and
 Lemma \ref{lem4.2} that for $R\in(0,\frac{1}{2}]$ we have
 $$u(x_j +\sqrt{r_j}\eta,t_j +r_j \zeta)\leq\frac{C}{r^{n/2}_{j}}\left[\frac{1}{R^{n/2}}+(N_{4R}f_j )(\eta,\zeta)\right]\quad\text{ for }(\eta,\zeta)\in{\cal P}_{2R} (0,0)$$ and
 $$v(x_j +\sqrt{r_j}\xi,t_j +r_j
 \tau)\leq\frac{C}{r^{n/2}_{j}}\left[\frac{1}{R^{n/2}}+(N_{2R}g_j
   )(\xi,\tau)\right]\quad\text{ for }(\xi,\tau)\in{\cal P}_R (0,0)$$ where
 $C$ is independent of $(\xi,\tau)$, $(\eta,\zeta)$, $j$, and $R$.
 It therefore follows from \eqref{eq4.34}$_2$ and \eqref{eq4.37} that for $R\in(0,\frac{1}{2}]$ we have
 \begin{align}\label{eq4.52}
  \notag r^{-\frac{n+2}{2}}_{j} f_j (\xi,\tau)&=(Hu)(x_j +\sqrt{r_j}\xi,t_j +r_j \tau)\\
  \notag &\leq C\left(\frac{1}{r^{n/2}_{j}}\left[\frac{1}{R^{n/2}}+(N_{2R}g_j )(\xi,\tau)\right]\right)^\lambda \\
  &\leq Cr^{-n\lambda/2}_{j}\left[\frac{1}{R^{n\lambda/2}}+(N_{2R}g_j )(\xi,\tau)^\lambda \right]\quad\text{ for }(\xi,\tau)\in{\cal P}_R (0,0)
 \end{align}
 and
 \begin{align*}
  \notag r^{-\frac{n+2}{2}}_{j} g_j (\eta,\zeta)&=(Hv)(x_j +\sqrt{r_j}\eta,t_j +r_j \zeta)\\
  \notag &\leq C\left(\frac{1}{r^{n/2}_{j}}\left[\frac{1}{R^{n/2}}+(N_{4R}f_j )(\eta,\zeta)\right]\right)^\sigma \\
  &\leq Cr^{-n\sigma/2}_{j}\left[\frac{1}{R^{n\sigma/2}}+(N_{4R}f_j )(\eta,\zeta)^\sigma \right]\quad\text{ for }(\eta,\zeta)\in{\cal P}_{2R} (0,0).  
 \end{align*}
 Thus for $(\xi,\tau)\in\mathbb{R}^n \times\mathbb{R}$ we have
 \begin{align*}
  ((N_{2R}g_j )(\xi,\tau))^\lambda &\leq C\left(r^{\frac{n+2}{2}-\frac{n\sigma}{2}}_{j}N_{2R}\left[\frac{1}{R^{n\sigma/2}}+(N_{4R}f_j )^\sigma \right](\xi,\tau)\right)^\lambda \\
  &\leq Cr^{\frac{(n+2-n\sigma)\lambda}{2}}_{j}\left[R^{(1-n\sigma/2)\lambda}+((M_{4R}f_j )(\xi,\tau))^\lambda \right] 
 \end{align*}
 where $M_R f :=N_R ((N_R f )^\sigma )$.  Hence by \eqref{eq4.52}
 there exists a positive constant $a$ which depends only on $n$,
 $\lambda$, and $\sigma$ such that for $R\in(0,\frac{1}{2}]$ we have
 \begin{equation}\label{eq4.53}
  f_j (\xi,\tau)\leq C\frac{1}{(Rr_j )^a}\left(1+((M_{4R}f_j )(\xi,\tau))^\lambda \right)\quad\text{ for }(\xi,\tau)\in{\cal P}_R (0,0).
 \end{equation}
 By \eqref{eq4.36} there exists $\varepsilon=\varepsilon(n,\lambda,\sigma)\in(0,1)$ such that
 \begin{equation}\label{eq4.54}
  \sigma<\frac{n+2}{n+\varepsilon}\quad\text{ and }\sigma<\frac{2-\varepsilon}{n+\varepsilon}+\frac{n+2}{n+\varepsilon}\frac{1}{\lambda}.
 \end{equation}
 To show that \eqref{eq4.42} cannot hold for all $\gamma>n+2$ and thereby complete the proof of (ii), it suffices by the definition of $r_j$ and $f_j$ to show for some $\gamma>0$ that the sequence
 \begin{equation}\label{eq4.55}
  \{r^{\gamma}_{j}f_j (0,0)\} \text{ is bounded.}
 \end{equation}

To prove \eqref{eq4.55} we need the following result.

\begin{lem}\label{lem4.6}
  Suppose the sequence
  \begin{equation}\label{eq4.56}
   \{r^{\alpha}_{j}f_j \} \text{ is bounded in }L^p ({\cal P}_{4R}(0,0))
  \end{equation}
  for some constants $\alpha\geq0$, $p\in[1,\infty)$ and $R\in(0,\frac{1}{2}]$.  Let $\beta=\alpha\lambda\sigma+a$ where $a$ is as in \eqref{eq4.53}.  Then either the sequence
  \begin{equation}\label{eq4.57}
   \{r^{\beta}_{j}f_j \} \text{ is bounded in }L^\infty ({\cal P}_R (0,0))
  \end{equation}
  or there exists a positive constant $C_0 =C_0 (n,\lambda,\sigma)$ such that the sequence
  \begin{equation}\label{eq4.58}
   \{r^{\beta}_{j}f_j \} \text{ is bounded in }L^q ({\cal P}_R (0,0))
  \end{equation}
  for some $q\in(p,\infty)$ satisfying
  \begin{equation}\label{eq4.59}
   \frac{1}{p}-\frac{1}{q}>C_0 .
  \end{equation}
 \end{lem}

\begin{proof}
 It follows from \eqref{eq4.53} that
 \begin{equation}\label{eq4.60}
  r^{\beta}_{j}f_j (\xi,\tau)\leq\frac{C}{R^a}\left(1+(M_{4R}(r^{\alpha}_{j}f_j ))(\xi,\tau))^\lambda \right)\quad\text{ for }(\xi,\tau)\in{\cal P}_R (0,0).
 \end{equation}
 We can assume
 \begin{equation}\label{eq4.61}
  p\leq\frac{n+2}{2}
 \end{equation}
 for otherwise from Theorem \ref{HeatPotential} and \eqref{eq4.56} we find that the sequence
 $\{N_{4R}(r^{\alpha}_{j}f_j )\}$ is bounded in $L^\infty ({\cal
   P}_{4R}(0,0))$ and hence by \eqref{eq4.60} we see that \eqref{eq4.57} holds.
 
 Define $p_2$ by
 \begin{equation}\label{eq4.62}
  \frac{1}{p}-\frac{1}{p_2}=\frac{2-\varepsilon}{n+2}
 \end{equation}
 where $\varepsilon=\varepsilon(n,\lambda,\sigma)$ is as in \eqref{eq4.54}.
 By \eqref{eq4.61}, $p_2 \in(p,\infty)$ and by Theorem \ref{HeatPotential} we have
 \begin{equation}\label{eq4.63}
  \|(N_{4R}f_j )^\sigma \|_{p_2 /\sigma}=\| N_{4R}f_j \|^{\sigma}_{p_2}\leq C\| f_j \|^{\sigma}_{p}
 \end{equation}
 where $\|\cdot\|_p:=\|\cdot\|_{L^p ({\cal P}_{4R}(0,0))}$.  Since, by \eqref{eq4.54}, 
 $$\frac{1}{p_2}=\frac{1}{p}-\frac{2-\varepsilon}{n+2}\leq
 1-\frac{2-\varepsilon}{n+2}=\frac{n+\varepsilon}{n+2}<\frac{1}{\sigma}$$ 
we have
 \begin{equation}\label{eq4.64}
  \frac{p_2}{\sigma}>1.
 \end{equation}
 We can assume 
 \begin{equation}\label{eq4.65}
p_2 /\sigma\le (n+2)/2 
\end{equation}
for otherwise by Theorem \ref{HeatPotential} and \eqref{eq4.63} we have
 $$\| M_{4R}(r^{\alpha}_{j}f_j )\|_\infty 
\leq C\|(N_{4R}(r^{\alpha}_{j}f_j ))^\sigma \|_{p_2/\sigma}\leq C\| r^{\alpha}_{j}f_j \|^{\sigma}_{p}$$ which is bounded by \eqref{eq4.56}.  Hence by
 \eqref{eq4.60} we see that \eqref{eq4.57} holds.

Define $p_3$ and $q$ by 
\begin{equation}\label{eq4.66}
\frac{\sigma}{p_2}-\frac{1}{p_3}=\frac{2-\varepsilon}{n+2}
\quad\text{and}\quad q=\frac{p_3}{\lambda}.
\end{equation}
By \eqref{eq4.64} and \eqref{eq4.65}, $p_3\in(1,\infty)$ and by
Theorem \ref{HeatPotential}
\begin{align*}
  \Vert \left(M_{4R}f_j\right)^\lambda \Vert_q &=\Vert M_{4R}f_j\Vert^{\lambda}_{p_3}\\
  &\leq C\Vert(N_{4R}f_j)^\sigma \Vert^{\lambda}_{p_2/\sigma} 
\leq C\Vert f_j \Vert^{\lambda\sigma}_{p}
 \end{align*}
by \eqref{eq4.63}. It follows therefore from \eqref{eq4.60} that 
\[
\|r_j^\beta f_j\|_{L^q({\cal P}_R(0,0))}
\le\frac{C}{R^a}\left(1+\|r_j^\alpha f_j\|_p^{\lambda\sigma}\right)
\]
which is a bounded sequence by \eqref{eq4.56}. 
It remains to prove that $q$ satisfies \eqref{eq4.59} for some positive
constant $C_0=C_0(n,\lambda, \sigma)$.

By \eqref{eq4.62} and \eqref{eq4.66} we have 
 \begin{align}\label{eq4.67}
  \notag \frac{1}{p}-\frac{1}{q}&=\frac{1}{p}-\frac{\lambda}{p_3}
=\frac{1}{p}+\frac{(2-\varepsilon)\lambda}{n+2}-\frac{\lambda\sigma}{p_2}\\
  \notag &=\frac{1}{p}+\frac{(2-\varepsilon)\lambda}{n+2}+\frac{(2-\varepsilon)\lambda\sigma}{n+2}-\frac{\lambda\sigma}{p}\\
  &=-\frac{\lambda\sigma-1}{p}+\frac{(2-\varepsilon)\lambda\sigma+(2-\varepsilon)\lambda}{n+2}.
 \end{align}
 \begin{description}
  \item[Case I.] Suppose $\lambda\sigma\leq 1$.  Then by
    \eqref{eq4.67}, \eqref{eq4.36}, 
 and \eqref{eq4.40} we get
  \begin{equation*}
   \frac{1}{p}-\frac{1}{q}\geq \frac{(2-\varepsilon)\lambda\sigma+(2-\varepsilon)\lambda}{n+2} \geq C_1(n)>0.
  \end{equation*}
  \item[Case II.] Suppose $\lambda\sigma>1$.  Then, by \eqref{eq4.67},
  \begin{align*}
   \frac{1}{p}-\frac{1}{q}&\geq 1-\sigma\lambda+
\frac{(2-\varepsilon)\lambda\sigma+(2-\varepsilon)\lambda}{n+2}\\
   &=\frac{1}{n+2} [n+2+(2-\varepsilon)\lambda-\lambda\sigma(n+2-(2-\varepsilon))]\\
   &=\frac{(n+\varepsilon)\lambda}{n+2} 
\left[ \frac{2-\varepsilon}{n+\varepsilon}+\frac{n+2}{n+\varepsilon} \frac{1}{\lambda}-\sigma \right]\\
   &=C_2(n,\lambda,\sigma)>0
  \end{align*}
  by \eqref{eq4.54}.  
\end{description}
Thus \eqref{eq4.59} holds with $C_0=\min(C_1,C_2)$.
This completes the proof of Lemma \ref{lem4.6}.
\end{proof}

We return now to the proof of Lemma \ref{lem4.5}(ii).  By
\eqref{eq4.46}, the sequence $\{f_j\}$ is bounded in $L^1({\cal P}_2(0,0))$.
Starting with this fact and iterating Lemma \ref{lem4.6} a finite
number of times ($m$ times is enough if $m>1/C_0$) we see that there
exists $R_0 \in(0,\frac{1}{2})$ and $\gamma>n$ such that sequence
$\{r_j^\gamma f_j\}$ is bounded in $L^\infty({\cal P}_{R_0}(0,0))$. In
particular \eqref{eq4.55} holds. This completes the proof of Lemma
\ref{lem4.5}(ii).
\end{proof}

\section{Proofs}\label{sec5}
In this section we prove the results in Section \ref{sec2}.

\begin{proof}[Proof of Theorem \ref{thm2.1}]
  Since increasing $\sigma$ and/or $\lambda$ weakens the conditions
  \eqref{eq2.1} on $u$ and $v$, we can assume $\sigma=\lambda=\frac{n+2}{n}$.
  Let $w=u+v$.  Then it follows from \eqref{eq2.1} that in
  $\Omega\times(0,1)$ we have
 \begin{align*}
  0&\le Hw=Hu+Hv
\le \left(v+\left(\frac{1}{\sqrt{t}}\right)^n\right)^{\frac{n+2}{n}}
+\left(u+\left(\frac{1}{\sqrt{t}}\right)^n\right)^{\frac{n+2}{n}}\\
&\le C\left(w+\left(\frac{1}{\sqrt{t}}\right)^n\right)^{\frac{n+2}{n}}
 \end{align*}
 for some positive constant $C$. Thus by \cite[Theorem 1.1]{T2011}
for each compact subset $K$ of $\Omega$ we have 
\begin{equation*}
 \max_{x\in K} (u(x,t)+v(x,t))=\max_{x\in K} w(x,t)
= O\left(\left(\frac{1}{\sqrt{t}}\right)^n\right) 
\quad \quad\text{ as }t\to0^+
 \end{equation*}
 which proves \eqref{eq2.6} and \eqref{eq2.7}.
\end{proof}

\begin{proof}[Proof of Theorem \ref{thm2.2}] Since increasing $\sigma$ weakens
  the conditions on $u$ and $v$ in \eqref{eq2.1} but does not change
  the estimates \eqref{eq2.9} and \eqref{eq2.10}, we can, instead of
  \eqref{eq2.8}, assume
 \begin{equation}\label{eq5.6}
  \lambda>\frac{n+2}{n}\quad\text{ and }\quad\frac{2}{n}<\sigma
<\frac{2}{n}+\frac{n+2}{n\lambda}.
 \end{equation}
Let $\{\Omega_i \}$ be a sequence of bounded open subsets of
 $\mathbb{R}^n$ such that
 $$\overline{\Omega}_i \subset\Omega\quad\text{ and }K\subset\overline{\Omega}_{i+1}\subset\Omega_i \quad\text{ for }i=1,2,... .$$
 By Lemma \ref{lem4.5}(ii), for some $\gamma>n+2$ we
 have 
$$\max_{x\in\overline{\Omega}_1}Hu(x,t)
=O\left(\left(\frac{1}{\sqrt{t}}\right)^\gamma \right)\quad\text{ as }t\to0^+ .$$ 
Hence by Lemma \ref{lem4.4}
 \begin{equation}\label{eq5.7}
  \max_{x\in\overline{\Omega}_2}v(x,t)=O\left(\left(\frac{1}{\sqrt{t}}\right)^p \right)\quad\text{ as }t\to0^+
 \end{equation}
 for some $p>n$.  Thus by \eqref{eq2.1},
 $$\max_{x\in\overline{\Omega}_2}Hu(x,t)
\leq
O\left(\left(\frac{1}{\sqrt{t}}\right)^{p\lambda}\right)
\quad\text{ as }t\to 0^+.$$  
Thus by Lemma \ref{lem4.4} we get
 \begin{equation}\label{eq5.8}
  \max_{x\in\overline{\Omega}_3}v(x,t)=O\left(\left(\frac{1}{\sqrt{t}}\right)^n
  \right)+o\left(\left(\frac{1}{\sqrt{t}}\right)^{p\lambda\frac{n\sigma-2}{n+2}}\right)
\quad\text{as }t\to 0^+.
 \end{equation}
 By \eqref{eq5.6}, $\lambda\frac{n\sigma-2}{n+2}<1$.  Thus iterating a
 finite number of times the procedure of going from \eqref{eq5.7} to
 \eqref{eq5.8} we obtain for some positive integer $k$ that
 \begin{equation}\label{eq5.9}
  \max_{x\in\overline{\Omega}_k}v(x,t)=O\left(\left(\frac{1}{\sqrt{t}}\right)^n \right)\quad\text{ as }t\to0^+
 \end{equation}
 which clearly implies \eqref{eq2.10}.  It follows from \eqref{eq5.9}
 and $\eqref{eq2.1}_1$ that
 $$\max_{x\in\overline{\Omega}_k}Hu(x,t)
=O\left(\left(\frac{1}{\sqrt{t}}\right)^{n\lambda}\right)\quad\text{ as }t\to0^+ .$$  Thus by Lemma \ref{lem4.4} and \eqref{eq5.6}
 $$\max_{x\in\overline{\Omega}_{k+1}}u(x,t)
=o\left(\left(\frac{1}{\sqrt{t}}\right)^{\frac{n^2}{n+2}\lambda}\right)\quad\text{ as }t\to0^+$$ which clearly implies \eqref{eq2.9}.
\end{proof}

\begin{proof}[Proof of Theorem \ref{thm2.3}] Since
  $\lambda>\tfrac{n+2}{n}$ and $\varphi(t)\to 0^+$ as $t\to 0^+$ there
  exists a sequence $\{T_j\}\subset\mathbb{R}$ such
  that $0<4T_{j+1}<T_j <\tfrac{1}{2}$,
\[
\sum^{\infty}_{j=1}\varepsilon_j <\infty\quad\text{ where }\varepsilon_j 
=\sqrt{\varphi(T_j )},
\] and
\[
0<r_j <T_j /2 \quad\text{ where }r_j =T^{\tfrac{n\lambda}{n+2}}_{j}.
\]
Let
\[
M=\min_{\overline{{\cal P}_{1/2}(0,1)}}\Phi
\] 
where ${\cal P}_r(x,t)$ is defined by \eqref{eq4.9}.  Then $M>0$ and
\begin{equation}\label{eq5.5}
  \min_{\overline{{\cal P}_{T_j/2}(0,T_j)}}\Phi=M/T^{n/2}_{j}.
\end{equation}
For the rest of this proof the variables $(x,t)$ and $(\xi,\tau)$
(resp. $(y,s)$ and $(\eta,\zeta)$) will be related by
\[
x=\sqrt{r_j}\xi,\,t=T_j +r_j \tau \quad\text{ (resp.
}y=\sqrt{r_j}\eta,\,s=T_j +r_j \zeta).
\]  
Under this change of variables,
$$(y,s)\in{\cal P}_{r_j}(0,T_j )
\text{ if and only if }(\eta,\zeta)\in{\cal P}_1 (0,0).$$  Let 
$\psi:\mathbb{R}^n \times\mathbb{R}\to[0,1]$ be a $C^\infty$ 
function whose support is $\overline{{\cal P}_1 (0,0)}$.  Define 
$\psi_j :\mathbb{R}^n \times\mathbb{R}\to[0,1]$ by
$$\psi_j (y,s)=\psi(\eta,\zeta).$$
Then the support of $\psi_j$ is $\overline{{\cal P}_{r_j}(0,T_j )}$ and 
$$\iint_{\mathbb{R}^n \times\mathbb{R}}\psi_j (y,s)\,dy\,ds
=\iint_{\mathbb{R}^n
  \times\mathbb{R}}\psi(\eta,\zeta)r^{\tfrac{n}{2}+1}_{j}d\eta\,
d\zeta=r^{\tfrac{n+2}{2}}_{j}I$$ 
where 
$$I=\iint_{\mathbb{R}^n \times\mathbb{R}}\psi(\eta,\zeta)\,d\eta\,d\zeta>0.$$  
Let
$$f=\sum^{\infty}_{j=1}M_j \psi_j \quad\text{ where }M_j 
=\frac{\varepsilon_j}{r^{\tfrac{n+2}{2}}_{j}}.$$  
Since the functions $\psi_j$ have disjoint supports, 
$f\in C^\infty ((\mathbb{R}^n \times\mathbb{R})\backslash\{0,0\})$.  
Also
$$\iint_{\mathbb{R}^n \times\mathbb{R}}f(y,s)\,dy\,ds
=\sum^{\infty}_{j=1}M_j \iint_{\mathbb{R}^n \times\mathbb{R}}\psi_j(y,s)\,dy\,ds
=I\sum^{\infty}_{j=1}M_j r^{\tfrac{n+2}{2}}_{j}
=I\sum^{\infty}_{j=1}\varepsilon_j <\infty.$$  
Thus the
functions $u,v:\mathbb{R}^n \times\mathbb{R}\to[0,\infty)$ defined by 
\begin{align*}
u(x,t)&=\iint_{\mathbb{R}^n \times\mathbb{R}}\Phi(x-y,t-s)f(y,s)\,dy\,ds\\
v(x,t)&=\frac{1}{M}\Phi(x,t)
\end{align*} 
are $C^\infty$ on $\mathbb{R}^n \times\mathbb{R}\backslash\{(0,0)\}$
and they clearly satisfy \eqref{eq2.11}$_2$ and \eqref{eq2.14}.

For $(x,t)\in\overline{{\cal P}_{r_j}(0,T_j )}$ we have
\begin{align*}
 u(x,t)&\geq\iint_{{\cal P}_{r_j}(0,T_j )}\Phi(x-y,t-s)M_j \psi_j (y,s)\,dy\,ds\\
 &=\frac{\varepsilon_j}{r^{n/2}_{j}}\iint_{{\cal P}_1 (0,0)}
\Phi(\xi-\eta,\tau-\zeta)\psi(\eta,\zeta)\,d\eta \,d\zeta.
\end{align*}
Thus, letting
$$J=\iint_{{\cal P}_1
  (0,0)}\Phi(-\eta,-\zeta)\psi(\eta,\zeta)\,d\eta \,d\zeta>0$$ 
we find that
\[
 u(0,T_j )\ge\frac{\varepsilon_j J}{r^{n/2}_{j}}
 =\frac{\sqrt{\varphi(T_j)}J}{T^{\tfrac{n^2 \lambda}{2(n+2)}}_{j}}>>\frac{\varphi(T_j )}{T^{\tfrac{n^2\lambda}{2(n+2)}}_{j}}\quad\text{ as }j\to\infty
\]
which proves \eqref{eq2.13}.

Also, for $(x,t)\in{\cal P}_{r_j}(0,T_j )$, it follows from \eqref{eq5.5} that 
\begin{align*}
 Hu(x,t)&=f(x,t)=M_j \psi_j \leq M_j\\
 &=\frac{\varepsilon_j}{r^{\tfrac{n+2}{2}}_{j}}=\frac{\varepsilon_j}{T^{\tfrac{n\lambda}{2}}_{j}}\leq\left(\frac{1}{T^{n/2}_{j}}\right)^\lambda \\
 &\leq\left(\frac{1}{M}\Phi(x,t)\right)^\lambda =v(x,t)^\lambda
\end{align*}
which yields \eqref{eq2.11}$_1$.
\end{proof}

\begin{proof}[Proof of Theorem \ref{thm2.4}]
Define 
\[
p:=\frac{\lambda+1}{\lambda\sigma-1}\quad\text{and}
\quad q:=\frac{\sigma+1}{\lambda\sigma-1}.
\]
Then 
\[
\frac{1}{p}-\frac{2}{n}=\frac{\lambda}{\lambda+1}
\left[\sigma-\left(\frac{2}{n}+\frac{n+2}{n\lambda}\right)\right].
\]
Thus by \eqref{eq2.15}
\begin{equation}\label{eq5.10}
0<q\le p<n/2.
\end{equation}
Also
\begin{equation}\label{eeq5.11}
\lambda q=p+1\quad \text{and}\quad \sigma p=q+1.
\end{equation}
Let $\{T_j\}\subset (0,1)$ be a sequence such that $T_j\to 0$ as
$j\to\infty$. Define $w_j,z_j:(-\infty,T_j)\to(0,\infty)$ by 
\[
w_j(t)=(T_j-t)^{-p}\quad\text{and}\quad z_j(t)=(T_j-t)^{-q}.
\]
Then by \eqref{eq5.10} and \eqref{eeq5.11}, we have for $0\le t<T_j$ that
\begin{equation}\label{eeq5.12}
w_j(t)\ge  z_j(t), \quad w_j'(t)\ge  z_j'(t), 
\quad w_j'(t)=pz_j(t)^\lambda,  \quad z_j'(t)=qw_j(t)^\sigma.
\end{equation}
Choose $t_j\in (0,T_j)$ such that $w_j(t_j)=t_j^{-n/2}$. Then
\begin{equation}\label{eeq5.13}
\frac{T_j}{t_j}-1=t_j^{\frac{n}{2p}-1}\to 0\quad\text{as } j\to\infty
\end{equation}
by \eqref{eq5.10}.

Choose $a_j\in (t_j,T_j)$ such that $z_j(a_j)>j\varphi(a_j)$. Then
\begin{equation}\label{eq5.11}
\frac{z_j(a_j)}{\varphi(a_j)}\to\infty\quad \text{as}\quad j\to\infty.
\end{equation}
Let $h_j(s) = \sqrt{4(a_j-s)}$ and $H_j(s) = \sqrt{4(a_j+\vp_j-s)}$ where 
$\vp_j>0$ satisfies 
\begin{equation}\label{eq5.12}
a_j+2\vp_j<T_j,\quad t_j-\vp_j>t_j/2,\quad w_j(t_j-\vp_j) > 
\frac{w_j(t_j)}2, \quad \text{and}\quad z_j(t_j-\vp_j) > 
\frac{z_j(t_j)}2.
\end{equation}
Define
\begin{align*}
\omega_j &= \{(y,s) \in \R^n\times \R\colon \ |y|<h_j(s)\quad 
\text{and}\quad t_j<s<a_j\}\\
\Omega_j &= \{(y,s) \in \R^n\times \R\colon \ |y|<H_j(s) \quad 
\text{and}\quad t_j-\vp_j<s<a_j+\vp_j\}.
\end{align*}
By taking a subsequence, we can assume the sets $\Omega_j$ are
pairwise disjoint.

Let $\chi_j\colon \ \R^n\times\R\to [0,1]$ be a $C^\infty$ function 
such that $\chi_j\equiv 1$ in $\omega_j$ and $\chi_j\equiv 0$ 
in $\R^n\times \R\setminus\Omega_j$. 
Define $f_j,g_j,u_j,v_j\colon \R^n \times \R\to 
[0,\infty)$ by
\[
f_j(y,s) = \chi_j(y,s)w'_j(s),\quad g_j(y,s) = \chi_j(y,s)z'_j(s) 
\]
\[
u_j(x,t) = \iint_{\R^n\times\R} \Phi(x-y,t-s)
f_j(y,s)\,dy\,ds
\]
and
\[
v_j(x,t) = \iint_{\R^n\times\R} \Phi(x-y,t-s) g_j(y,s)\,dy\,ds.
\]
Then $u_j$ and $v_j$ are $C^\infty$ and 
\begin{equation}\label{eq5.13}
\begin{array}{ll}
Hu_j=f_j,\quad Hv_j=g_j&\text{in } \R^n \times 
\R\\
\noalign{\medskip}
u_j=v_j=0&\text{in } \R^n\times (-\infty,0)\end{array}
\end{equation}
where $Hu=u_t-\Delta u$ is the heat operator.

By \eqref{eeq5.12} and Theorem \ref{HeatPotential} we have
\begin{align}
\left\|~\iint_{\Omega_j\setminus\omega_j} \Phi(x-y,t-s) 
z'_j(s)\,dy\,ds\right\|_{L^\infty(\R^n\times (0,1))} 
&\le \left\|~\iint_{\Omega_j\setminus\omega_j} \Phi(x-y,t-s) 
w'_j(s)\,dy\,ds\right\|_{L^\infty(\R^n\times (0,1))}\nonumber\\
&\le C_n\|w'_j(s)\|_{L^{n+2}(\Omega_j\setminus\omega_j)}\nonumber\\
&\le z_j(t_j)\le w_j(t_j)\label{eq5.14}
\end{align}
provided we decrease $\vp_j$ if necessary.

Also, for $(x,t)\in\Omega_j$ we have
$
|x|\le \sqrt{4(T_j-t_j)}
$
by \eqref{eq5.12}; and thus using \eqref{eq5.12} again we obtain
\begin{equation}\label{eq5.15}
\max_{(x,t)\in\Omega_j} \frac{|x|^2}t \le \frac{4(T_j-t_j)}{t_j-\vp_j} \le 
\frac{8(T_j-t_j)}{t_j}\to 0\quad \text{as}\quad j\to\infty
\end{equation}
by \eqref{eeq5.13}. Hence there exists a positive number $M$,
independent of $j$, such that for all $(x,t)\in\Omega_j$ we have
\begin{equation}\label{M}
M\Phi(x,t)\ge 2/t_j^{n/2}=2w_j(t_j)\ge 2z_j(t_j). 
\end{equation}

In order to obtain a lower bound for $u_j$ and $v_j$ in $\Omega_j$,
note first that for $t_j-\vp_j\le s\le t\le a_j+\vp_j$ and $|x|\le
H_j(t)$ we have
\begin{align}
\label{eq5.16}
\int_{|y|<H_j(s)} \Phi(x-y,t-s)\,dy &= \frac1{\pi^{n/2}} 
\int_{|z-\frac{x}{\sqrt{4(t-s)}}|<\frac{H_j(s)}{\sqrt{4(t-s)}}} e^{-|z|^2} dz\\
\label{eq5.17}
&\ge \frac1{\pi^{n/2}} \int_{|z-\frac{H_j(s)e_1}{\sqrt{4(t-s)}}| < 
\frac{H_j(s)}{\sqrt{4(t-s)}}} e^{-|z|^2} dz \quad \text{where }  e_1 = 
(1,0,\ldots,0)\\
\label{eq5.18}
&\ge \alpha_n
\end{align}
where
\begin{equation}\label{eq5.19}
\alpha_n := \frac1{\pi^{n/2}} \int_{|z-e_1|<1} e^{-|z|^2} \, dz\in (0,1).
\end{equation}
Some of the steps in the above calculation need some explanation. Equation 
\eqref{eq5.16} is obtained by making the change of variables 
$z=\frac{x-y}{\sqrt{4(t-s)}}$. Since $|x| \le H_j(t)\le H_j(s)$, the center of 
the ball of integration in \eqref{eq5.16} is closer to the origin than the 
center of the ball of integration in \eqref{eq5.17}. Thus, since the integrand 
$e^{-|z|^2}$ is a decreasing function of $|z|$, we obtain \eqref{eq5.17}. Since 
$H_j(s) \ge \sqrt{4(t-s)}$, the ball of integration in \eqref{eq5.17} contains 
the ball of integration in \eqref{eq5.19} and hence inequality \eqref{eq5.18} 
holds.

Using \eqref{eq5.18} and \eqref{eq5.19}, we find for $(x,t)
\in\Omega_j$ that
\begin{align*}
\iint_{\Omega_j} \Phi(x-y,t-s) w'_j(s)\,dy\,ds
&= \int^t_{t_j-\vp_j} w'_j(s) \left(~\int_{|y|<H_j(s)}\Phi(x-y,t-s)\,dy\right)ds\\
&\ge \alpha_n(w_j(t) -w_j(t_j-\vp_j)) \ge \alpha_n w_j(t) - w_j(t_j)
\end{align*}
and similarly
\[
\iint_{\Omega_j} \Phi(x-y,t-s) z'_j(s)\,dy\,ds
\ge \alpha_n z_j(t) - z_j(t_j).
\]
It therefore follows from \eqref{eq5.14} 
that for $(x,t) \in\Omega_j$ we have
\begin{align}
u_j(x,t) &\ge \iint_{\omega_j} \Phi(x-y,t-s) w'_j(s)\,dy\,ds\nonumber\\
&= \iint_{\Omega_j} \Phi(x-y,t-s) w'_j(s)\,dy\,ds -\iint_{\Omega_j 
\setminus\omega_j} \Phi(x-y,t-s) w'_j(s)\,dy\,ds\nonumber\\
\label{eq5.20}
&\ge \alpha_nw_j(t) -2w_j(t_j)
\end{align}
and similarly
\[
v_j(x,t)\ge \alpha_nz_j(t) -2z_j(t_j).
\]

Also,
\begin{align*}
\iint_{\R^n\times\R} f_j(y,s)\,dy\,ds 
&\le \iint_{\Omega_j} w'_j(s)\,dy\,ds\\
&\le p\int^{T_j}_0 (T_j-s)^{-(p+1)}\left(\int_{|y|<\sqrt{4(T_j-s)}}dy\right) ds\\
&= \omega_np \int^{T_j}_0 (T_j-s)^{-(p+1)}(4(T_j-s))^{n/2}ds\\
&=  4^{n/2}\omega_np \int^{T_j}_0 (T_j-s)^{n/2-p-1}ds\\
&= 4^{n/2}\omega_np \int^{T_j}_0 \tau^{n/2-p-1}\,d\tau\\
&\to 0 \quad \text{as }j\to\infty
\end{align*}
by \eqref{eq5.10}.
We consequently obtain from \eqref{eeq5.12} that
\[
\iint_{\R^n\times\R} \sum^\infty_{j=1}g_j(y,s)\,dy\,ds 
\le \iint_{\R^n\times\R} \sum^\infty_{j=1} f_j(y,s)\,dy\,ds < \infty
\]
provided we take a subsequence if necessary. Hence the functions
$u,v\colon (\R^n\times \R)\setminus \{(0,0)\}\to [0,\infty)$
defined by
\begin{equation*}
\begin{aligned}
u(x,t)&=1+M\Phi(x,t) + \sum^\infty_{j=1} u_j(x,t)\\
v(x,t)&=1+M\Phi(x,t) + \sum^\infty_{j=1} v_j(x,t)
\end{aligned}
\end{equation*}
are $C^\infty$ and by \eqref{eq5.13} we have
\begin{alignat}{2}\label{eq5.21}
Hu &= \sum^\infty_{j=1}f_j,\quad &Hv= \sum^\infty_{j=1}g_j
\qquad &\text{in } (\R^n\times\R)\setminus\{(0,0)\}\\ 
u &= 0,\quad &v= 0\qquad &\text{in } \R^n\times (-\infty,0).\nonumber
\end{alignat}
Also, for 
$(x,t)\in\Omega_j$ we have by \eqref{eq5.20} and \eqref{M} that
\begin{align}
u(x,t) &\ge M\Phi(x,t) +u_j(x,t)\nonumber\\
&\ge M\Phi(x,t) + (\alpha_nw_j(t)-2w_j(t_j))\nonumber\\
\label{eq5.22}
&\ge \alpha_nw_j(t)
\end{align}
and similarly
\begin{equation}\label{eq5.23}
v(x,t)\ge\alpha_nz_j(t).
\end{equation}
Thus by \eqref{eeq5.12} and \eqref{eq5.11} we have 
\begin{align*}
\min\left\{\frac{u(0,a_j)}{\varphi(a_j)},
\frac{v(0,a_j)}{\varphi(a_j)}\right\}
&\ge \min\left\{\frac{\alpha_nw_j(a_j)}{\varphi(a_j)},
\frac{\alpha_nz_j(a_j)}{\varphi(a_j)}\right\}\\
&=\frac{\alpha_nz_j(a_j)}{\varphi(a_j)}\to\infty\quad\text{as }j\to\infty,
\end{align*}
and so $u$ and $v$ satisfy \eqref{eq2.17} and \eqref{eq2.18}.

It also follows from \eqref{eeq5.12}, \eqref{eq5.21}, \eqref{eq5.22}
and \eqref{eq5.23} that for $(x,t)\in\Omega_j$ we have
\begin{equation}
\begin{aligned}
Hu(x,t) &= f_j(x,t) \le w'_j(t) = pz_j(t)^\lambda 
\le p\left(\frac{v(x,t)}{\alpha_n}\right)^\lambda\\
Hv(x,t) &= g_j(x,t) \le z'_j(t) = qw_j(t)^\sigma 
\le q\left(\frac{u(x,t)}{\alpha_n}\right)^\sigma.
\end{aligned} \label{eq5.24}
\end{equation}
Inequalities \eqref{eq5.24} also hold for $(x,t) \in (\R^n\times
\R) \setminus\bigcup\limits^\infty_{j=1} \Omega_j$ because
$Hu=Hv=0$ there by \eqref{eq5.21}.  We thus obtain inequalities
\eqref{eq2.16} by scaling the independent variables $x$ and $t$.
\end{proof}

\begin{proof}[Proof of Theorem \ref{thm2.5}]
Theorem \ref{thm2.5} follows from (and is actually the same as) Lemma
\ref{lem4.5}(i).
\end{proof}

\begin{proof}[Proof of Theorem \ref{thm2.6}]
Theorem \ref{thm2.6} follows immediately from the conclusion
\eqref{eq4.22} in Lemma \ref{lem4.4}.
\end{proof}

\end{document}